\documentclass[11pt]{amsart}
\usepackage{amsmath,amsthm,amsfonts,amssymb}

\usepackage{hyperref}
\usepackage[all]{xy}
\usepackage{graphicx}
\usepackage{amscd}
\usepackage{graphicx}
\usepackage{pb-diagram}
\usepackage{color}
\usepackage{enumerate}
\usepackage{fourier}
\usepackage{mathtools}

\hypersetup{
    colorlinks=true,
    linkcolor=blue,
    citecolor=green,
    urlcolor=cyan
}

\numberwithin{equation}{section}

\theoremstyle{plain} 
\newtheorem{theorem}{Theorem}[section]
\newtheorem{corollary}[theorem]{Corollary}
\newtheorem{lemma}[theorem]{Lemma}
\newtheorem{prop}[theorem]{Proposition}

\theoremstyle{definition}
\newtheorem{definition}[theorem]{Definition}

\theoremstyle{remark}
\newtheorem{remark}[theorem]{Remark}

\newcommand{\PP}{\mathbb{P}}

\newcommand{\R}{\mathbb{R}}
\newcommand{\C}{\mathbb{C}}
\newcommand{\Z}{\mathbb{Z}}

\newcommand{\LL}{\mathbb{L}}
\newcommand{\bL}{\mathbb{L}}

\newcommand{\OC}{\mathcal{OC}}
\newcommand{\CA}{\mathcal{A}}

\newcommand{\ks}{\mathfrak{ks}}

\newcommand{\str}{\mathrm{Str}}
\newcommand{\be}{\mathbf{1}}

\newcommand{\tr}{\mathrm{Tr}}

\newcommand{\AI}{A_\infty}

\newcommand{\Hom}{{\rm Hom}}

\newcommand{\bx}{\boldsymbol{x}}

\begin{document}

\title[]{Pairings in mirror symmetry \\ between a Symplectic manifold and a Landau-Ginzburg $B$-model}
\author{Cheol-hyun Cho}
\author{Sangwook Lee}
\author{Hyung-Seok Shin}
\address{Department of Mathematical Sciences, Research Institute in Mathematics, Seoul National University, Gwanak-gu, Seoul, South Korea}
\email{chocheol@snu.ac.kr}
\address{Korea Institute for Advanced Study\\ 85 Hoegi-ro Dongdaemun-gu, Seoul 02455, Korea}
\email{swlee@kias.re.kr}
\address{Korea Institute for Advanced Study\\ 85 Hoegi-ro Dongdaemun-gu, Seoul 02455, Korea}
\email{hsshin@kias.re.kr}

\begin{abstract}
We find a relation between Lagrangian Floer pairing of a symplectic manifold and Kapustin-Li pairing of the mirror Landau-Ginzburg model under localized mirror functor.  They are conformally equivalent with an interesting conformal factor $(vol^{Floer}/vol)^2$, which can be described as a ratio of Lagrangian Floer volume class and classical volume class. 

For this purpose,  we introduce $B$-invariant of Lagrangian Floer cohomology with values in Jacobian ring of the mirror potential function. And we prove what we call a multi-crescent Cardy identity under certain conditions, which is
a generalized form of Cardy identity.

As an application, we discuss the case of general toric manifold, and the relation to the work of Fukaya-Oh-Ohta-Ono and their $Z$-invariant.  Also, we compute the  conformal factor $(vol^{Floer}/vol)^2$  for the elliptic curve quotient $\mathbb{P}^1_{3,3,3}$, which gives a modular form.
\end{abstract}
\maketitle

\tableofcontents

\section{Introduction}
Mirror symmetry between a symplectic manifold $M$ and a Landau-Ginzburg $B$-model of a holomorphic function $W$ has been actively investigated  both in open and closed string sector.
For open string $A$-model, we consider Fukaya category of $M$, and for open string $B$-model, we consider the dg-category of matrix factorizations of $W$. For a closed string $A$-model, we restrict ourselves to quantum cohomology of $M$, and for $B$-model, Jacobian ring of $W$.

Pairings in each of these play important roles. Let us recall the pairings in closed theories.
Quantum cohomology  $QH^*(M)$ becomes a Frobenius algebra with a Poincar\'e duality pairing, and this generalizes to the structure of Frobenius manifold (see Dubrovin \cite{Du}). Jacobian ring of $W$ has a residue pairing, which generalizes to the Saito's flat structure on  the space of universal unfoldings of $W$ \cite{Sa}. The closed string mirror symmetry  have been studied by many people,  Givental, Fukaya-Oh-Ohta-Ono and Iritani to name a few. In particular, Fukaya-Oh-Ohta-Ono introduced a geometric construction of a map from quantum cohomology to the Jacobian ring, called Kodaira-Spencer map $\ks$.  One could ask whether this Kodaira-Spencer map preserves the pairing structures.  For example, Fukaya-Oh-Ohta-Ono \cite{FOOOtoric3} have shown that if the toric manifold is nef and the mirror potential function is Morse, then Kodaira-Spencer map preserves the (closed) pairings (i.e. $\ks$ is an isometry).

We will find that $\ks$ is not an isometry in general, but only a conformal map whose  conformal factor
can be described in terms of Lagrangian Floer theory.

Let us describe the pairings in open string theories. Fukaya category of a compact symplectic manifold has cyclic symmetric inner product. This is a pairing on Lagrangian Floer theory, given by Poincar\'e duality, with further cyclic symmetry with respect to $\AI$-operations (see Fukaya \cite{Fu}, cf. Ganatra \cite{G} for noncompact cases).
For matrix factorization category of a Landau-Ginzburg model, such pairing is given by Kapustin-Li pairing, suggested by Kapustin-Li \cite{KL} and mathematically developed by several people, including Murfet \cite{Mur}, Dyckerhoff-Murfet \cite{DM}, Polishchuck-Vaintrob \cite{PV10} and Shklyarov \cite{Shk}.
More precisely, the Kapustin-Li pairing is given by the following
formula. For $f\in \Hom_{MF(W)}(X,Y)$ and $g\in \Hom_{MF(W)}(Y,X[n])$,
\begin{equation}\label{kl}
  \langle f, g \rangle_{KL} =  \frac{1}{(2\pi i)^n n!}\oint_{\{|\partial_i W|=\epsilon\}}\frac{ \str (fg(dQ_Y)^{\wedge n})}{\partial_1 W \partial_2 W \cdots \partial_n W}
  \end{equation}
where $Q_Y$ is the structure map of the matrix factorization $Y$ (i.e. $Q_Y^2  = W \cdot Id$).

These open-string structures are sometimes called Calabi-Yau structures. Costello  \cite{Cos} have shown that the structure of open-closed topological conformal field theory is equivalent to a Calabi-Yau $\AI$-category. The relationship of pairings between open and closed theories in $A$-model (or in $B$-model)  are present already on the level of 2 dimensional topological conformal field theory (see Section 2 for more details).   Also, Kontsevich-Soibelman \cite{KS}, Costello \cite{Cos} and Ganatra-Perutz-Sheridan \cite{GPS} use this structure to build up closed theory from open string theory.

We are interested in relations of pairing under homological mirror symmetry between Fukaya category of $M$ and matrix factorization category of $W$.
So far, such a relation between open string pairings was largely mysterious.
We investigate such a relationship in this paper and  we use this to find the conformal factor for
Kodaira-Spencer map of closed string mirror symmetry.

Now, let us explain the construction in this paper.
We wish to establish the following commutative diagram in the case of toric manifolds.
Let $L$ be a Lagrangian submanifold of $M$, and denote by $HF(L,L)$ its Lagrangian Floer cohomology (which
we assume to be well-defined possibly after adding a bounding cochain $\xi$).
\begin{equation}\label{diagram7}   \xymatrix{
	HF(L,L) \ar[dd]^{\mathcal{F}^\bL}  \ar[rr]^-{OC} \ar@{-->}[ddrr]^-{B} &  \;\;\;\;   &  QH^*(M) \ar[d]^{\ks}  \\
	& & Jac(W_\bL) \ar[d]^-{I} \\
	MF(\mathcal{F}^\bL(L),\mathcal{F}^\bL(L))    \ar[rr]_-{\textrm{boundary-bulk}} &   \;\;\;\;   &   \;\;\;\;    Jac(W_\bL)dx_1 \cdots dx_n}
\end{equation}
Recall that open-closed map on $A$-side (denoted by $OC$) is constructed by counting appropriate
pseudo-holomorphic disks. On $B$-side, the corresponding open-closed map is more often called {\em boundary-bulk map}.
For closed string mirror symmetry, we consider the Kodaira-Spencer map $\ks$ (as in \cite{FOOOtoric3}).
(Geometric construction of  Kodaira-Spencer map from quantum cohomology to Jacobian ring is only known in the
case of toric manifolds \cite{FOOOtoric3}.  There is a work in progress by Amorim-Hong-Lau and the first author in the case of orbispheres.)
For open string mirror symmetry,  we use the homological mirror functor $\mathcal{F}^\bL$ from localized mirror construction
of \cite{CHL} and \cite{CHL2}. 
There an explicit homological mirror functor has been constructed by studying formal deformation space of
a particular (immersed) Lagrangian submanifold $\bL$, which is called reference Lagrangian. Floer potential of $\bL$
defines a Landau-Ginzburg model $W_\bL$ defined on the formal Maurer-Cartan space of $\bL$.
Then, a curved version of Yoneda embedding canonically defines an explicit $\AI$-functor from entire Fukaya category to the matrix factorization category of $W_\bL$. It is called localized mirror functor, since the functor only detects Lagrangians which has non-trivial Floer intersection with $\bL$. In particular, it provides an explicit correspondence between the Floer strip and the structure map $Q$ of
the matrix factorization ($Q^2 = W_\bL - \lambda$ for some constant $\lambda$).
Localized mirror formalism works a little differently depending on whether the reference Lagrangian $\bL$ is an immersed Lagrangian or Lagrangian torus (see Section \ref{sec:lmf}.)

Now, to compare $(\ks \circ OC)$ and $(\textrm{boundary-bulk} \circ \mathcal{F}^\bL)$, we introduce an isomorphism
$$I :  Jac(W_\bL) \to  Jac(W_\bL)dx_1 \cdots dx_n$$ defined by 
$$ I (f) = c_\bL \cdot f dx_1 \cdots dx_n,\;\;\;  c_\bL \in \Lambda_0 $$
Here, $c_\bL = vol^{Floer}/vol$ is the ratio of Floer volume class and classical volume class defined  in Definition \ref{defn:fv}.

The isomorphism $I$ can be viewed from the following perspective. 
Kodaira-Spencer map $\ks$ is a ring isomorphism to Hochschild {\em cohomology} of $MF(W_\LL)$,
which is isomorphic to the Jacobian ring $Jac(W_\LL)$. On the other hand, the residue is canonically defined on $Jac(W_\LL)dx_1\cdots dx_n$, which is isomorphic to the Hochschild {\em homology} of $MF(W_\LL)$. 
Therefore, the issue is how to identify Hochschild cohomology and homology as modules.
We propose to use the map $I$ above and this factor comes from the  Diagram \ref{diagram7}.
\begin{figure}
	\includegraphics{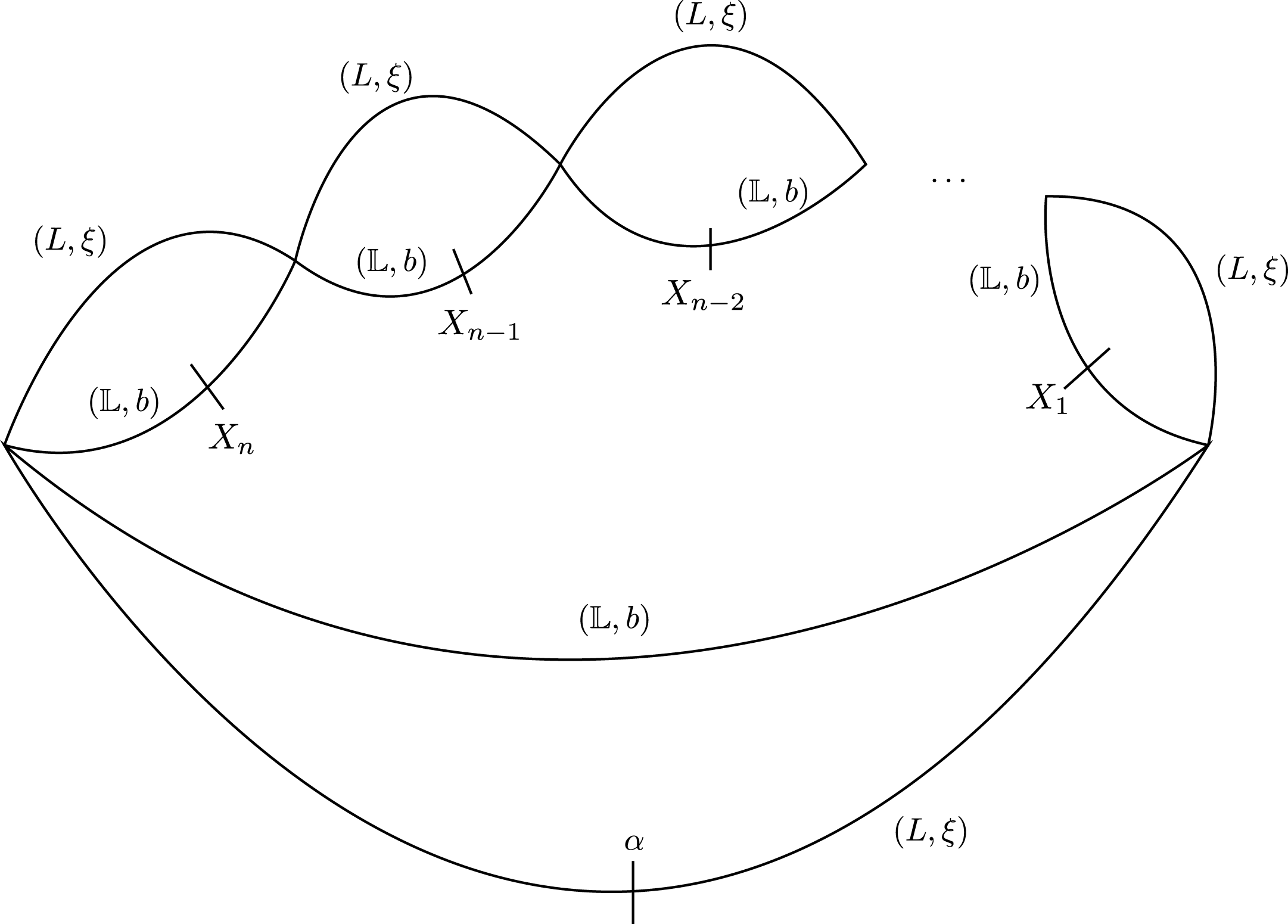}
	\caption{Boundary-bulk map of  $\mathcal{F}^\bL(\alpha)$}
	\label{fig:supertrace}
\end{figure}

\begin{figure}
\includegraphics{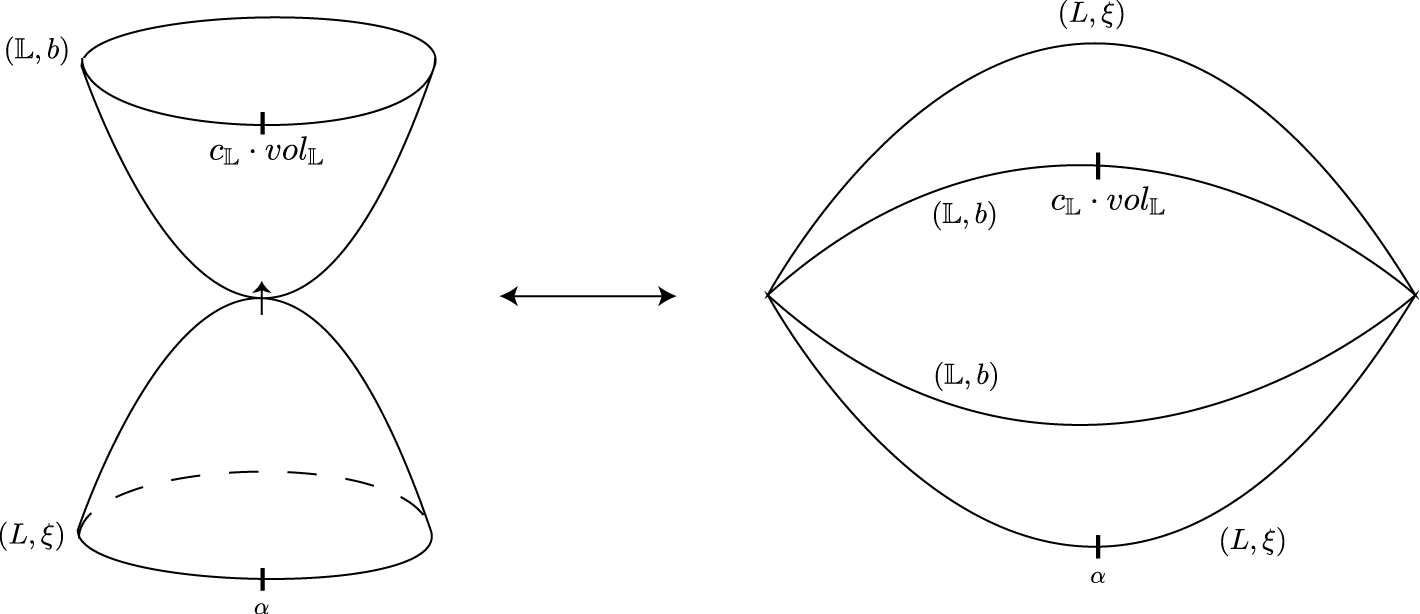}
\caption{Cardy identity between $I \circ \ks(\OC(\alpha))$ and $B$}
\label{fig:cobordism}
\end{figure}
One of main observation is that the map $(\textrm{boundary-bulk} \circ \mathcal{F}^\bL)$, which is the supertrace $\str \big( \mathcal{F}^\bL(\alpha) \circ (dQ)^{\wedge n}\big)$ in \eqref{kl},
corresponds to the diagram of $J$-holomorphic discs in Figure \ref{fig:supertrace} (multi-crescent diagram).
Then geometrically,  the commutativity of the diagram \eqref{diagram7}  corresponds to
identification of Figure \ref{fig:supertrace} and Figure \ref{fig:cobordism}.
Here, $B$-invariant corresponds to the diagram on the right of  Figure \ref{fig:cobordism} (bi-crescent diagram) and
the  upper-right composition  $I\circ\ks \circ OC$ corresponds to the diagram on the  left of Figure \ref{fig:cobordism} (sand-glass diagram). The correspondence in Figure \ref{fig:cobordism} is called Cardy identity.
Thus, we may call the commutativity of the diagram \eqref{diagram7}  a {\em multi-crescent Cardy identity}. 


 \begin{theorem}
\begin{enumerate}
\item The lower-left triangle of  \eqref{diagram7} always commutes.
\item The upper-right triangle of  \eqref{diagram7} commutes for toric manifolds.
\end{enumerate}
\end{theorem}
We prove (1)  in Theorem \ref{mainthm} and Corollary \ref{cor:ll}.
We find how to transform multi-crescents in Figure 
\ref{fig:supertrace} to bi-crescents $B(\alpha)$, using $\AI$-equations, Jacobian relations and properties of supertrace.

We prove (2) in Theorem \ref{thm:kscommute}.
The part (2) corresponds to Figure \ref{fig:cobordism}, which is a reformulation of the proof in  toric cases by
Fukaya-Oh-Ohta-Ono \cite{FOOOtoric3}.
Although cobordism argument was verified only for toric manifolds, we expect the theorem to hold in general cases.

In the case of orbi-sphere $\PP^1_{a,b,c}$, Amorim-Hong-Lau and the first author defines Kodaira-Spencer map
in the work in progress
and proves closed string mirror symmetry. The proof of Cardy identity in \cite{FOOOtoric3} carries over to this case also
(we will assume this in this paper).

Next step is to consider pairings in closed string mirror symmetry
with isomorphism $I$.

\begin{equation}\label{diagramfooo7}   \xymatrix{
	QH^*(M)_\LL  \otimes QH^*(M)_\LL \ar[r]^-{ \langle \cdot, \cdot \rangle_{PD_M}} \ar[d]^{\ks \otimes \ks} & \;\;\;\; \Lambda \;\;\;\; \ar@{=}[dd] \\
	  Jac(W_\LL) \otimes Jac(W_\LL) \ar[d]^{I\otimes I
	  } &   \\
	 \big(Jac(W_\LL)dx_1\cdots dx_n \big) \otimes \big( Jac(W_\LL)dx_1\cdots dx_n \big)\;\;\;\; \ar[r]^-{\langle \cdot, \cdot \rangle_{res}} &\;\;\;\; \Lambda \;\;\;\; }  
\end{equation}
This diagram means that the Kodaira-Spencer map preserves the pairings up to the conformal factor $(c_\bL)^2$.
Unfortunately, it is not known whether the diagram commutes in general.
There are cases that we know the above diagram commutes.
\begin{prop}
\begin{enumerate}
\item
Let $M$ be a general toric manifold such that the mirror $W$ has only Morse singularities.
Let $\LL$ be a Lagangian torus fiber.
Then, the above diagram \eqref{diagramfooo7} commutes.

\item
Let $M=\PP^1_{3,3,3}$, an elliptic orbifold sphere and $\mathbb{L}$ be the Seidel Lagrangian.
Then the above diagram \eqref{diagramfooo7} commutes.
\end{enumerate}
\begin{remark}
 In the toric case, LG mirror is written in $y_i$-variables with $y_i=\exp(x_i)$ and hence
the top form $dx_1\cdots dx_n$ is replaced by $\frac{dy_1}{y_1}\cdots \frac{dy_n}{y_n}$.
\end{remark}
\end{prop}
For $\PP^1_{3,3,3}$, we compute the residue of the Kodaira-Spencer map directly to prove the proposition
in Theorem \ref{thm:9}.

For toric manifolds,  we prove it in Proposition \ref{prop:m}. The above proposition in that case is a reformulation of the result of \cite{FOOOtoric3}.The isomorphism $I$ does not appear in their work, but we show that their result is equivalent to the above diagram: Fukaya-Oh-Ohta-Ono introduced the new residue pairing on $B$-side $\langle \cdot, \cdot \rangle_{res_Z} $ using Floer theoretic $Z$-invariant. We find that the difference of $\langle \cdot, \cdot \rangle_{res_Z} $ and the complex geometric residue pairing $\langle \cdot, \cdot \rangle_{res} $ is recorded in the isomorphism $I$.

We can combine diagram \eqref{diagram7} and diagram \eqref{diagramfooo7}. For this it is more
convenient to consider the traces instead.

\begin{equation}\label{firstdiagram}   \xymatrix{
	HF(L,K) \times HF(K,L) \ar[d]^{\mathcal{F}^\bL \otimes \mathcal{F}^\bL}  \ar[r]^-{product} & HF(L,L) \ar[r]^-{\int_L} \ar[d]^{\mathcal{F}^\bL} & \Lambda \ar[d]^{(c_\bL)^2 \cdot -} \\
	MF(\mathcal{F}^\bL(L),\mathcal{F}^\bL(K)) \times MF(\mathcal{F}^\bL(K),\mathcal{F}^\bL(L))   \ar[r]^-{product} & MF(\mathcal{F}^\bL(L),\mathcal{F}^\bL(L))  \ar[r]^-{KL} & \Lambda }
\end{equation}
 The first square commutes because the functor respects the product structures. 
 
As a corollary, we obtain the following comparison result.
\begin{theorem}
Let $M$ be a general toric manifold such that the mirror $W$ has only Morse singularities and $\bL$ a Lagrangian torus fiber.
(or $M = \PP^{1}_{3,3,3}$ and $\bL$ a Seidel Lagrangian.)
\begin{enumerate}
\item (Open pairings)
Lagrangian Floer pairing $\langle,\rangle_{PD_L}$ and Kapustin-Li pairing can be identified by multiplication of $c_\bL^2$.
i.e. for homological mirror functor $\mathcal{F}$, we have
$$ \langle v, w \rangle_{PD_L} =  c_\bL^2 \langle \mathcal{F}(v), \mathcal{F}(w) \rangle_{KL}$$
\item (Closed pairings)
Pairing for quantum cohomology and residue pairing for Jacobian ring
can be identified by multiplication of $c_\bL^2$.
$$ \langle v, w \rangle_{PD_M} = c_\bL^2 \langle \ks(v), \ks(w) \rangle_{Res}$$
\end{enumerate}
In $\PP^{1}_{3,3,3}$ case, $c_\LL$ is a modular form given by 
\[ c_\LL=\sum_{k \in \mathbb{Z}} (-1)^k q^{(6k+1)^2}.\]
For $M$ toric manifold which is nef and if the mirror potential $W$ is Morse, then we have $c_\LL=1$.
\end{theorem}
The conformal factor $c_\LL$ in $\PP^{1}_{3,3,3}$ case is computed in Lemma \ref{lem:cl}

We believe that the conformal factor $c_\bL$ is related to the choice of a primitive form of Saito \cite{Sa}
over the universal unfolding of $W_\bL$  (which can be obtained as a bulk-deformed version of the construction in this paper).  We hope to investigate the relationship in the future.

\section*{Acknowledgements}
We would like to thank Kaoru Ono, Atsushi Takahashi, Lino Amorim, Hansol Hong and Siu-Cheong Lau for helpful conversations. We would like to thank Dohyeong Kim at Seoul National University for the proof of mirror identity in Appendix \ref{app2}.
The work of C.H. Cho is supported by Samsung Science and Technology Foundation under Project Number SSTF-BA1402-05.
\section{Pairings in mirror symmetry}
\subsection{Open-closed maps and pairings}
In this subsection, we recall the expected properties of the open and closed pairings from the
point of view of open-closed 2-d topological Field theory.
We refer readers to  \cite{La}, \cite{Seg} for more details.
Recall that 2-d TFT is roughly given by a functor from 2-d cobordism category to the category of  (graded) vector spaces,
with the sewing condition.  For the open-closed theory, an oriented surface is allowed to have two type of boundaries. 
Let $\mathcal{B}_0$ be the set of boundary conditions.  Then, one type of boundary is either closed string boundary (a closed circle) or an open string boundary which is an oriented interval with labels in $\mathcal{B}_0$ at each end points. The other is a boundary sector which is a open or closed curve carrying a single label in $\mathcal{B}_0$.

Sewing operation is allowed only between two closed or open string boundaries, and the orientations (and labels for open case) of boundaries should match.
Sewing condition means that  that sewing operation of surfaces corresponds to a composition of  morphisms between vector spaces.

For a closed string boundary, TFT assigns a vector space $\mathcal{C}$, and for an open string boundary which is oriented from $a$ to $b$, TFT assigns
a vector space $\mathcal{O}_{ab}$. 
Pair of pants corresponds to the commutative product structure on $\mathcal{C}$, and 
 open version of pair of pants corresponds to the composition in a $\C$-linear category $\mathcal{B}$ whose objects correspond to $\mathcal{B}_0$. Both $\mathcal{C}$ and $\mathcal{O}_{aa}$ have non-degenerate traces
$$\tr_{cl}:\mathcal{C} \to k, \tr_a:\mathcal{O}_{aa} \to k,$$
which also gives rise to closed and open pairings
 $$\langle v, w \rangle_{cl} = \tr_{cl}(v \cdot w), \langle \phi_1,\phi_2 \rangle_{op} = \tr_a (\phi_2 \circ \phi_1) $$
 for $\phi_1 \in \mathcal{O}_{ab}, \phi_2 \in \mathcal{O}_{ba}$. Here non-degeneracy follows from the gluing axiom.
 Now, there are open-closed (boundary-bulk) and closed-open (bulk-boundary) maps
 $$ CO_a: \mathcal{C} \to \mathcal{O}_{aa}, OC^a:\mathcal{O} \to \mathcal{C}.$$
$CO_a$ is an algebra homomorphism preserving the identity element ($CO(\be_{cl})= \be_{op}$). $CO_a$ and $OC^a$ are adjoints to each other
\begin{equation}\label{coad}
\langle v , OC^a(\phi) \rangle_{cl} = \langle CO_a(v),\phi \rangle_{op}
\end{equation}
By setting $v =\be_{cl}$ in the above adjoint identity, we obtain the relation between open and closed traces:
\begin{equation}\label{trco}
 \tr_{cl}(OC^a(\phi)) = \tr_{a} (\phi) 
 \end{equation}
These properties give $\mathcal{C}$ (resp. $\mathcal{B}$) a structure of Frobenius algebra (resp. a Calabi-Yau category).

Recall that 2-d topological conformal field theory where one additionally consider chain model of the moduli space of bordered Riemann surfaces
gives rise to the data of Calabi-Yau $\AI$-category (see  Kontsevich-Soibelman \cite{KS}, Costello \cite{Cos}).
Here Calabi-Yau structure is given by a pairing, which is compatible with $\AI$-operation.
Such a structure is an important ingredient to reconstruct closed theory from open theory.
We remark that there is more general notion of a (non-compact) Calabi-Yau $\AI$-category by Kontsevich-Soibelman by $A^! \cong A[n]$ where $A^!$ is the inverse dualizing bimodule of $A$.
We refer readers to Kontsevich-Soibleman \cite{KS},  Ganatra \cite{G} for more detailed constructions.

%
%
%



Now, let us recall  $A$-model for symplectic manifold, and $B$-model for Landau-Ginzburg model,
and relevant structures that we have discussed.
\subsection{$A$-model}
For a compact symplectic manifold $M$,  closed string $A$-model is given by quantum cohomology of $M$.
As a module it is isomorphic to singular cohomology $H^*(M;\Lambda)$ but its cup product structure is deformed by $J$-holomorphic curves. 
The associated pairing is Poincare duality pairing and thus its closed trace is given by the integration 
$$\tr_{cl} = \int_M : H^{2n}(M,\Lambda) \to \Lambda.$$

Open $A$-model of $M$ is given by Fukaya category $\mathcal{F}u(M)$ which is a filtered $\AI$-category. Its objects are  compact Lagrangian submanifolds, which are possibly immersed relatively spin, and equipped with flat complex line bundle on it. For simplicity, we denote each object by the Lagrangian $L$.
 Its morphisms $Hom(L_1,L_2)$ are given by Lagrangian Floer complex $CF(L_1,L_2)$ and $m_k$ operations are defined by counting $J$-holomorphic polygons with Lagrangian boundary conditions.
 For its construction we refer readers to \cite{FOOO} and \cite{Seidelbook}.

The open trace $\tr_L: HF^*(L,L) \to \Lambda$ is  given by a composition
$$ HF^*(L,L)  \stackrel{PSS}{\longrightarrow} H^*(L,\Lambda) \stackrel{\int_L}{\longrightarrow} \Lambda$$
where  PSS denotes well-known Piunikhin-Salamon-Schwarz isomorphism (when it exists).
This gives rise to  homologically non-degenerate, graded symmetric pairing (see \cite{Seidelbook} (12e)). The graded symmetry of the pairing follows from the graded symmetry of the product.

Closed-open and open-closed maps have been defined  in this setting and they have been proved to be very useful
(see  \cite{BC}, \cite{FOOO}, \cite{Se1} for example).
In the simplest setting, we have
$$OC_L: HF^*(L,L) \to QH^*(M), CO^L:QH^*(M) \to HF^*(L,L)$$
These maps can be defined by considering the moduli space $\mathcal{M}_{1,1}(M,L,\beta)$ of  $J$-holomorphic maps from a disk with one interior $z^+$ and one boundary marked point $z_0$
of homology class $\beta$. If $z^+$ is an input and $z_0$ is an output marking, then we obtain $CO$-map, and the other way around for $OC$-map.
If one uses the same perturbation for these two moduli spaces, then we should have the adjoint property of $CO$ and $OC$ \eqref{coad}.
See \cite{BC} Theorem 2.1.1 for the case of  monotone Lagrangians.  For general toric manifolds, perturbations are actually different (called $p$ and $q$-perturbation) but still this adjoint property is proved in \cite{FOOOtoric3} Theorem 3.3.8.

In \cite{Fu}, Fukaya has constructed a cyclic filtered $\AI$-algebra for a Lagrangian $L$.

%
%
%

\subsection{Laundau-Ginzburg $B$-model}
Consider $R = \C[x_1,\cdots,x_n]$ or $\C[x_1,x_1^{-1},\cdots,x_n,x_n^{-1}]$.
Given a Landau-Ginzburg potential $W \in R$,  closed string Landau-Ginzburg $B$-model is given by Jacobian ring 
$$Jac(W) = \frac{R}{( \partial_{x_1} W, \cdots,  \partial_{x_n} W)}$$
In this case,  trace $\tr_{cl}$ is given by the residue pairing
$$\tr_{cl}(f)= \frac{1}{(2\pi i)^n} \oint_{\{|\partial_{x_i}W| = \epsilon_i\}} \frac{f dx_1\wedge \cdots \wedge dx_n}{ \partial_{x_1} W \cdots  \partial_{x_n} W}, \;f \in Jac(W).$$
If $W$ has only non-degenerate critical points, then residue can be easily computed to be
$$ \tr_{cl}(f) = \sum_{ z \in crit(W)} \frac{f(z)}{Hess (W(z))}$$

Open string Laundau-Ginzburg $B$-model is 
the dg-category of matrix factorizations.
\begin{definition}\label{def:mf}
Fix $\lambda \in \C$.
 A {\em matrix factorization} of $W-\lambda$ is a $\Z/2$-graded trivial vector bundle $E = E^0 \oplus E^1$
 on $Spec(R)$ with an endomorphism 
 \[ Q= \left[ {\begin{array}{cc} 0 & Q_{01} \\ Q_{10} & 0 \ \end{array} } \right], \; \textrm{where} \; Q_{ij} \in \Hom(E^j,E^i), \]
 satisfying $Q^2 = (W-\lambda) \cdot Id$.  In other words, $Q_{01}Q_{10} = (W-\lambda) \cdot Id_{E^0}$ and $Q_{10}Q_{01} = (W-\lambda) \cdot Id_{E^1}$.
 \end{definition}
 
Matrix factorizations of $W-\lambda$ form a differential $\Z/2$-graded category $MF_\lambda(W)$ as follows: 
\begin{definition}
Given two matrix factorizations $(E,Q), (F,Q')$, $\Z/2$-graded morphisms from $(E,Q)$ to $(F,Q')$ are
given by homomorphisms 
 \[ \Phi= \left[ {\begin{array}{cc} \Phi_{00} & \Phi_{01} \\ \Phi_{10} & \Phi_{11} \ \end{array} } \right], \; \textrm{where} \; \Phi_{ij} \in \Hom(E^j,F^i). \]
Compositions of morphisms 
$$\Hom\big((E_1,Q_1),(E_2,Q_2)\big) \times \Hom\big((E_0,Q_0),(E_1,Q_1)\big) \to \Hom\big((E_0,Q_0),(E_2,Q_2)\big)$$
 are defined in the obvious way.
The differential on a morphism is defined as
$$ \delta \Phi = Q \Phi  - (-1)^{|\Phi|} \Phi Q$$
for morphisms of homogeneous degrees.
These form a differential graded category $MF_\lambda(W)$.
\end{definition}
Its derived category $[MF_\lambda(W)]$ is equivalent to the derived category of singularity of $W$,  $D_{sg}(W^{-1}(
\lambda))$ by Orlov \cite{Or}.
Here, $D_{sg}(Z)$ is defined by the quotient of $D^b(Z)$ by $Perf(Z)$ which is a thick subcategory consisting of complexes which are quasi-isomorphic to bounded complexes of projectives. If $Z$ is smooth then $D_{sg}(Z)$ is trivial, hence the category measures how the variety is singular.

The pairing for LG $B$-model was suggested by  by Kapustin and Li \cite{KL}.
\begin{definition}
For $f\in \Hom_{[MF(W)]}(X,Y)$ and $g\in \Hom_{[MF(W)]}(Y,X[n])$,
\[  \langle f, g \rangle_{KL} =  \frac{1}{(2\pi i)^n n!}\oint_{\{|\partial_{x_i} W|=\epsilon\}}\frac{\mathrm{Str}(fg(dQ_Y)^{\wedge n})}{\partial_{x_1} W \partial_{x_2} W \cdots \partial_{x_n} W}\]
where $Q_Y$ is the structure map of the matrix factorization $Y$.
\end{definition}
Here $$(dQ)^{\wedge n} := (dQ)^{\wedge n} = (\partial_1Q \cdot dx_1 + \cdots + \partial_nQ \cdot dx_n)^{\wedge n}$$ which is given by matrix multiplication of matrices of one forms via wedge products. We refer readers to Appendix B for the definition and properties of $\mathrm{Str}$.

Many people have contributed for the mathematical understanding of the above pairing.
The existence of non-degenerate pairing without explicit description was established by Auslander \cite{Aus}
and Murfet \cite{Mur} gave a mathematical derivation of the above formula and the proof of non-degeneracy. 

We have $OC_Y: Hom_{MF(W)}(Y,Y) \to Jac(W)dx_1\wedge \cdots \wedge dx_n$ defined by
 $$ OC_Y(f) = \mathrm{Str}(f(dQ_Y)^{\wedge n})$$
Thus we have the compatibility of closed and open string traces under boundary-bulk map  \eqref{trco}.
Segal\cite{Seg}, Polishchuk-Vaintrob \cite{PV10} and Dyckerhoff-Murfet\cite{DM}   deepened the understanding 
boundary-bulk (open-closed) map, and we refer readers to these references for more details.
%
%

So far, $R$ is a (Laurant) polynomial algebra, but we need to consider more general algebras
with Novikov field coefficients:
$$R = \Lambda \ll x_1,\cdots,x_n \gg , R = \Lambda \ll x_1, x_1^{-1},\cdots,x_n,x_n^{-1} \gg$$
Here, $\ll,\gg$ denotes the completion with respect to the  filtration of Novikov field $\Lambda$.
In this case, we follow Dyckerhoff-Murfet \cite{DM} and use Grothendieck residues in the above definitions.
In \cite{DM}, they proved that these generalizations provide the relevant pairings in TCFT.

\section{Localized mirror functors}\label{sec:lmf}
In this section, we recall the formalism of localized mirror functors developed in \cite{CHL} and \cite{CHL2}.
Later we will use this to understand Kapustin-Li pairing in terms of Lagrangian Floer theory.
Here is the rough idea of the formalism. Fix a reference object $\bL$ in a filtered $\AI$-category $\mathcal{C}$, and consider its deformation theory. Namely, solve its weak Maurer-Cartan equation, and obtain a potential function $W_\bL$ on the solution space.
Then the first author with Hong and Lau \cite{CHL} defined a canonical $\AI$-functor from the original $\AI$-category $\mathcal{C}$ to the dg category of matrix factorizations of $W_\bL$. This construction may be considered as curved  Yoneda embedding with respect to $\bL$ and its weak Maurer-Cartan element. This can be generalized to the non-commutative Landau-Ginzburg models and quivers as discussed in \cite{CHL3}. 

\subsection{Deformation by Maurer-Cartan elements}
We briefly review the Maurer-Cartan formalism from Fukaya-Oh-Ohta-Ono \cite{FOOO} to set the notations.
Maurer-Cartan elements may be understood as a formal deformation of a given object. The idea of localized
mirror functor is to look at formal deformation space of a reference Lagrangian object $\bL$, and view other objects through these deformations. 

Consider a filtered unital $\AI$-category $\CA$ over $\Lambda$.
Novikov ring $\Lambda_0$ is defined as usual
\begin{equation*}
\Lambda_0 := \left\{ \sum_{j=0}^{\infty} c_j T^{\lambda_j} \left| c_j (\neq 0),\lambda_j \in \R, \lambda_j \geq 0, \lim_{j\to\infty}\lambda_j \nearrow +\infty\right.\right\}
\end{equation*}
For each $\lambda \in \R$, $\Lambda_0$ has a filtration
\begin{equation*}
F^{\lambda}\Lambda_0 := \left\{ \left.\sum c_j T^{\lambda_j} \in \Lambda_0 \right| \lambda_j \geq \lambda \right\}.
\end{equation*}
We denote $\Lambda_+ = F^{>0}\Lambda_0$ and the localization of $\Lambda_0$ gives the Novikov field $\Lambda$.

For any object $L$, we have a closed element $\be_L \in \Hom^0(L,L)$ such that
$$m_2(\be_L,w) = (-1)^{|w|}m_2(w,\be_L) = w, m_{k\neq2}(\cdots,\be_L,\cdots)=0$$
We can consider the following weak Maurer-Cartan equation 
\begin{equation}\label{mc}
 m(e^b) = m_0(1)+ m_1(b)+ m_2(b,b) + m_3(b,b,b) + \cdots = W(b) \cdot \be_L,
\end{equation}
for $b \in F^+ \Hom(L,L)$.
Given a weak Maurer-Cartan element $b$, one can deform the given $\AI$-structure on $Hom^*(L,L)$ by
$$m_k^b(w_1,\cdots,w_k) = m(e^b ,w_1 ,e^b ,\cdots,  e^b, w_k, e^b).$$
Here, the insertion of $e^b = 1 + b + b\otimes b + \cdots$ means  taking all possible summations over
the arbitrary number of insertions of $b$'s.
Note that $b$ solves weak Maurer-Cartan equation if $m_0^b(1)$ is a multiple of unit, and hence
$\{m_k^b\}_{k \geq 1}$ forms an $\AI$-algebra themselves. In particular, $(m_1^b)^2=0$ and defines
a chain complex.

Given weak Maurer-Cartan elements $\xi_i$ from $Hom(L_i,L_i)$ for $i=0,\cdots,k$,
we can deform $\AI$-structure on $\CA$ by
$$m_k^{\xi_0,\cdots,\xi_k}(w_1,\cdots,w_k) =  m(e^{\xi_0}, w_1 ,e^{\xi_1}, \cdots,  e^{\xi_{k-1}}, w_k, e^{\xi_k})$$

Localized mirror formalism has been developed in two settings, when $\bL$ is an immersed Lagrangian, or when $\bL$ is a torus $T^n$.
(Mixed type can be defined easily by combining two approaches, but for simplicity we will consider these two types only).

\subsection{Immersed cases}
We recall the construction of \cite{CHL} to which we refer readers for details.
Later in this paper, we will consider the particular case of $\mathbb{P}^1_{3,3,3}$
as an example.
Let $\bL$ be an immersed Lagrangian object in the $\Z/2$-graded Fukaya category $\mathcal{F}u(M)$.

Let $\{X_1,...,X_n\}$ be odd immersed generators in $CF^{1}(\LL,\LL)$.
We introduce formal variables $x_1,\cdots,x_n$ (with values in $\Lambda_+$ or $\Lambda_0$ with suitable convergnece assumption)
and set 
$$b = x_1 X_1 + \cdots + x_n X_n.$$
Denote by $Y$ the solution space of weak Maurer-Cartan equation \eqref{mc} for $\LL$, $m(e^b) = W(b) \cdot \be_{\LL}$,
and this defines a potential function $W_{\LL}:Y \to \Lambda_0$.
For simplicity, let us assume $Y =(\Lambda_0)^n$ (i.e. any $x_0, \cdots, x_n \in \Lambda_0$ solves the weak Maurer-Cartan equation).
We have $W(b) \in R = \Lambda_0 \ll x_1,\cdots,x_n \gg$.

The mirror functor $\mathcal{F}^\bL$ on objects works as follows.
Given an unobstructed object $K$ (i.e. with $m_0^K = 0$), consider $Hom(K,\LL)= CF(K,\LL)$ given by the Floer complex, which is $\Z/2$-graded, together with deformed differential
$m_1^{0,b}$. The deformed $\AI$-equation provides $$(- m_1^{0,b})^2 = \{ W_{\LL}(b) - 0 \} \cdot Id$$
which can be interpreted as matrix factorization of $W_{\LL}$, where $-m_1^{0,b}$ is a matrix with entries in $R$,
regarded as an $R$-module homomorphism between R-modules generated by even or odd intersection points.

This can be generalized for any weakly unobstructed object $(K,b')$ with 
$$(-m_1^{b',b})^2 = \{ W_{\LL}(b)  - W_K(b') \} \cdot Id$$
In this case, $b'$ does not involve any variables, and hence $W_K(b') \in \Lambda$ is a constant,
and it provides matrix factorization of $W_{\LL}$ at value $W_K(b')$. 
This can be made into an $\AI$-functor as we will explain in the last part of this section.

\subsection{Toric cases}\label{sec:toriclm}
We recall the construction of \cite{CHL2} to which we refer readers for details.
In the toric case, the idea in the immersed case does not immediately generalize.
Main reason is due to the choice of mirror variable. 
First, in toric case, the standard mirror variable is obtained from SYZ perspective based on the choice of Lagrangian torus fiber $L(u)$ and its holonomy. To apply our approach, we fix a Lagrangian and we only use the (generalized) holonomy part as a mirror variable. The second, more serious issue is that  in this case, we need a new variable $y = e^x$ and matrix factorizations in $y$ variable, but naive generalization of the above idea produces some factorizations in $x$-variables, which does not
provide matrix factorizations in $y$-variables.
In \cite{CHL2} the following approach was introduced to overcome these difficulties.
The idea is to consider a flat line bundle over the Lagrangian, whose holonomy is concentrated near co-dimension one tori (which is
called hypertori).

Consider basis $E_i \in H_1(\LL,\Z)$ for $i=1,\cdots,n$ and dual basis $E_i^* \in H^1(\LL,\Z)$.
We choose closed one forms $\theta_i \in \Omega^1(\LL,\C)$ whose cohomology is given by $E_i^* \in H^1(\LL,\Z)\otimes \C$.
We consider a flat $\C$-line bundle $\LL \times \C$, with a choice of a connection (for $x_i \in \C^*$)
$$\nabla = d - 2\pi \sqrt{-1}\sum_{i=1}^n x_i \theta_i.$$

Then, its holonomy along closed loops in $\LL$ is given by $\rho^b$,
where we set  $$b = \sum_{i=1}^n x_i \theta_i \in H^1(\LL,\C)/H^1(\LL,\Z),$$
$$ \rho^{b} = \pi_1(\LL) \to \C^*, \rho^{b}(\gamma) = \exp\big(2\pi \sqrt{-1}(b,\gamma )\big).$$
We can also deform $m_{k,\beta}$ to $m_{k,\beta}^\rho$ for each $\beta \in \pi_2(M,\LL)$ as
$$m_{k,\beta}^\rho(w_1,\cdots,w_k) =m_{k,\beta}(w_1,\cdots,w_k)\rho^{b}(\partial \beta).$$
In fact, we can generalize this approach, and consider a flat $\Lambda_0$-bundle $L \times \Lambda_0$
and $b \in H^1(\LL,\Lambda_0)/H^1(\LL,\Z)$ for $x_i \in \Lambda_0$. Note that holonomy along any loop lies in $\Lambda_0^*$.
This allows us to take $x_i$ in Novikov ring $\Lambda_0$.

We will may still write the resulting $\AI$-operation as $m_k^b = m_k^{b,\cdots,b}$ since $\rho$ depends on $b$,
and we read holonomy along the entire boundary of a holomorophic disc. We may also use the notation $m_k^{0,\cdots,0,b}$
when we read holonomy along the last segment of the boundary  of a holomorphic polygon for $m_k$. In this way, we can make the toric cases parallel to the immersed cases.

Now, to define a mirror functor, we want to make sure that the deformation $m_1^{0,b}$ can be written in terms of 
exponential variables $y_i = e^{x_i}$ only.  We need to choose a good connection one form.
If Lagrangian $K$ transversely intersect with $\bL$ at finitely many points, then we can choose hyper-tori away from these
intersection points. If we choose connection one form $\theta_i$'s to be supported in a small neighborhood of hypertori, away 
from intersection points, then we can see that holonomy contribution for any $J$-holomorphic polygon can be written 
in $y_i$-variables. This is the idea. 

In particular, $(- m_1^{0,b})^2$ counts $J$-holomorphic strips, with holonomy contribution along the upper-boundary of the strip,
and this holonomy can be easily computed by counting the intersection number between the boundary of the strip and the hyper-tori's.
As a result, we get a potential function in $R = \Lambda \ll y_1, y_1^{-1},\cdots, y_n,y_n^{-1} \gg$, and
matrix factorizations of $W \in R$.

\subsection{$\AI$-functor}
After setting up Maurer-Cartan space and potential function,  localized mirror functor is given canonically from
algebraic machinery, which is similar to Yoneda embedding.

To write the functor, we turn dg-category into an $\AI$-category (cf. \cite{Sh}).
\begin{definition}
The $\AI$-category $\AI(MF(W))$ is defined as follows. Its objects are $\Z/2$-graded finite dimensional matrix factorizations of $W$, denoted as $(E,Q)$ as in Definition \ref{def:mf}.
$$\Hom_{\AI(MF(W))}((E,Q_E), (F,Q_F)): =  Hom_{MF(W)}(F,E).$$
$m_1, m_2$ are defined as 
$$m_1 (\Phi) := \delta (\Phi) = Q_E \circ \Phi - (-1)^{|\Phi|}\Phi \circ Q_{F}$$
$$m_2(\Phi,\Psi) = (-1)^{|\Phi|} \Phi \circ \Psi$$
All other $m_k$'s are set to be zero. These define an $\AI$-category.
\end{definition}

\begin{theorem}[\cite{CHL} Theorem 2.18]
We have an $\AI$-functor $\mathcal{F}^\bL$ from Fukaya $\AI$-category to $\AI(MF(W))$.
On the level of objects, it sends a weakly unobstructed Lagrangian $L$ with Maurer-Cartan element $\xi$ to
the matrix factorization $(E,Q)$, which is given by 
$$E:=\Hom\big((L,\xi), (\bL,b)\big), Q:=-m_1^{\xi,b}$$
On the level of morphisms, 
 $$\mathcal{F}^\bL_k: \bigotimes_{i=1}^k \Hom\big( (L, \xi_{i-1}), (L,\xi_i) \big)  \to \Hom \big( (L, \xi_0), (L,\xi_k) \big)$$ is defined as
 \begin{equation}\label{lmfunctor}
\mathcal{F}^\bL_k (x_1,\cdots, x_k) (\bullet) = m_{k+1}^{\xi_0,\xi_1,\cdots,\xi_{k-1},b}(x_1,\cdots,x_k,\bullet)
 \end{equation}
 These define an $\AI$-functor.
\end{theorem}
\begin{remark}
In \cite{CHL}, we considered $\Hom\big((\bL,b), (L,\xi)\big)$ instead. In the current convention,
we do not have additional signs in \eqref{lmfunctor}.
\end{remark}
\begin{proof}
We omit the superscript $b$'s in the proof below as it is canonically determined from the inputs.
It is enough to show that 
$$ \sum_a (-1)^{|\bx^{(1)}_a|'} \mathcal{F}^\bL( \bx^{(1)}_a \otimes m(\bx^{(2)}_a) \otimes \bx^{(3)}_a)   =  \sum_c m_2( \mathcal{F}^\bL(\bx_c^{(1)}), \mathcal{F}^\bL(\bx_c^{(2)}))  
+ m_1 (\mathcal{F}^\bL(\bx)) $$
The first term can be written as
$$(-1)^{|\bx^{(1)}_a|'} m(\bx^{(1)}_a \otimes m(\bx^{(2)}_a) \otimes \bx^{(3)}_a,\bullet)$$
The second term  can be written as
$$(-1)^{|\bx^{(1)}_c |' +1 } \mathcal{F}^\bL(\bx_c^{(1)}) \circ  \mathcal{F}^\bL(\bx_c^{(2)})
= (-1)^{|\bx^{(1)}_c |' +1 } m(  \bx_c^{(1)}, m( \bx_c^{(2)}, \bullet))$$
The third term  can be written as
$$- m_1 m_{k+1}(\bx,\bullet) + (-1)^{|\bx|'+1} m_{k+1} (\bx,m_1(\bullet))$$
Hence, this proves the claim.
\end{proof}

\section{$B$-invariant}\label{sec:B}
In this section, we define $B$-invariant which assigns to a Floer cohomology element $\alpha \in HF(L,L)$ an element
of Jacobian ring $Jac(W_\bL)$. 

The motivation to define such an invariant is to have the following commuting diagram, which will be shown in the next section.
\begin{equation}\label{diagram0}   \xymatrix{
	HF(L,L) \;\;\;\;  \ar[d]^{\mathcal{F}^\bL}   \ar[drr]^-{B} &  &   \\
	MF(P_L,P_L)   \;\;\;\;     \ar[rr]_-{\textrm{boundary-bulk}} &  &   \;\;\;\;    Jac(W_\bL)dx_1\wedge \cdots \wedge dx_n. }
\end{equation}

Here is the setting. Let $M$ be a symplectic manifold with its Fukaya category $\mathcal{F}u(M)$.
We assume that $\mathcal{F}u(M)$ is a unital filtered $\AI$-category. (If it is only homotopy unital, one can take
quasi-isomorphic $\AI$-category which is unital and apply the construction.)
Let $\bL$ be a fixed reference Lagrangian submanifold with a weak Maurer-Cartan space $Y$ and a potential function $W_\bL: Y \to \Lambda$. Let $L$ be any Lagrangian in $\mathcal{F}u(M)$.
As explained in the last section, for immersed cases we also allow $L = \bL$, but
for toric cases, we require $L \neq \bL$ (but we can also take $L =\phi_H(\bL)$ as long as
$L \pitchfork \bL$).

We further assume the following condition, which holds in all of our examples.
\begin{definition}\label{defn:fv}
We assume that $HF(\bL,\bL)$ is generated by degree one elements $X_1,\cdots,X_n$ 
and that the total rank of the cohomology is $2^n$.
In this case, the (degree $n$ part of ) quantum product of all generators, will be called a {\em Floer volume class } or Floer point class. Namely, we define
\[ [X_n,\cdots,X_1]:=m_2^{b}\big(X_n,m_2^{b}(X_{n-1},\cdots m_2^{b}(X_3,m_2^{b}(X_2,X_1))\cdots)\big). \]
and its degree $n$-part with sign correction is denoted as
\[ vol^{Floer}:=  (-1)^{n+1}[X_n,\cdots,X_1]^{(n)}\]
\end{definition}
For the sign, we consider the toric case. Note that  the cup product convention of \cite{FOOO} is
 \[m_{2,0}(a,b)=(-1)^{|a|(|b|+1)}a\cup b.\] Denote by $m_{2,0}$ the classical cup product part of $m_2$, then we have
\[ m_{2,0}(X_n,m_{2,0}(X_{n-1},\cdots,m_{2,0}(X_2,X_1))\cdots)=(-1)^{\frac{n(n+1)}{2}+1}X_n \cup \cdots \cup X_1=(-1)^{n+1}vol_L,\] 
where $vol_L = X_1\cup \cdots \cup X_n$. Therefore, it is natural to correct it by the sign $(-1)^{n+1}$.
\begin{lemma}\label{lem:fg} 
We have the following
\begin{enumerate}
\item
Floer cohomology $HF((\bL,b),(\bL,b))$ is non-trivial if and only if $\frac{\partial W}{\partial x_i} = 0, \forall i$.
\item If $HF(\bL,\bL)$ is generated by $X_1,\cdots,X_n$, so is $HF((\bL,b),(\bL,b))$.
\item $m_1^b([X_n,\cdots,X_1])$ lies in the Jacobian ideal (generated by $\frac{\partial W}{\partial x_i}$).
\end{enumerate}
\end{lemma}
We will give a proof of this lemma at the end of this section.

Let us define $B$-invariant of any Lagrangian $L$ using the Floer complex $CF(L,\bL)$.
Note that the Floer complex $CF(L,\bL)$ is a $\Z/2$-graded $\Lambda$-module generated by the intersection points
$L \cap \bL$. 
\begin{definition}
Let $\alpha$ be a Floer cycle in $CF(L,L)$.
 We define the {\em B-invariant} of $\alpha$, $B(\alpha)$ with values in  $Jac(W_\bL)$ as follows. 
\begin{equation}\label{eq:Bdef}
B(\alpha) := (-1)^{\frac{n(n+1)}{2}+1}  \str \big( \bullet \mapsto  (-1)^{(n-1)|\bullet|} m_2^{0,0,b}(\alpha,m_2^{0,b,b}(\bullet,[X_n,\cdots,X_1])) \big) dx_1 \wedge \cdots \wedge dx_n
\end{equation}
where $\bullet$ runs over the intersection points
in $L \cap \bL$.
\end{definition}
\begin{remark}
We do not assume that $\AI$-algebra is cyclic symmetric.
\end{remark}

From Lemma \ref{lem:stp} (1), $B(\alpha)$ vanishes if degree of $\alpha$ is not equal to $n$ mod $2$.
\begin{remark}
The above construction extends to the case that $L$ is equipped with a Maurer-Cartan element $\xi$,
and $\alpha \in HF((L,\xi),(L,\xi))$. We can define
\begin{equation}\label{eq:Bdefo}
B (\alpha) := (-1)^{\frac{n(n+1)}{2}+1}  \str \big( \bullet \mapsto  (-1)^{(n-1)|\bullet|} m_2^{\xi,\xi,b}(\alpha,m_2^{\xi,b,b}(\bullet,[X_n,\cdots,X_1])) \big)dx_1 \wedge \cdots \wedge dx_n
\end{equation}
As such extension is obvious, we assume that $\xi=0$ to simplify the exposition.
\end{remark}

\begin{theorem}\label{thm:bw}
The map $B : HF(L,L) \to Jac(W_\bL)dx_1\cdots dx_n$ is well-defined. i.e. it is independent of the choice of representative in
its cohomology class.
\end{theorem}
\begin{proof}
In the proof, we will omit the notation $b$ as it can be determined from the inputs.  We must show that 
\begin{equation}\label{inv1}  \str (-1)^{(n-1)|\bullet|} m_2(m_1(\gamma),m_2(\bullet,[X_n,\cdots,X_1]))=0.
\end{equation}
We assume $|m_1(\gamma)|=|[X_n,\cdots,X_1]|$, otherwise (\ref{inv1}) is  zero by degree reason. Then $|\gamma|=n-1$. Applying $\AI$-relations repeteadly, we have

\begin{align*}
(\ref{inv1}) =& \str (-1)^{(n-1)|\bullet|+1}m_1 m_2(\gamma,m_2(\bullet,[X_n,\cdots,X_1])) \\
&+\str(-1)^{(n-1)|\bullet|+|\gamma|'+1}m_2(\gamma,m_1 m_2(\bullet,[X_n,\cdots,X_1])) \\
=& \str(-1)^{(n-1)|\bullet|+1} m_1 m_2(\gamma,m_2(\bullet,[X_n,\cdots,X_1])) \\
 &+\str (-1)^{(n-1)|\bullet|+n} m_2 (\gamma,m_2(m_1(\bullet),[X_n,\cdots,X_1])) \\
&+ \str(-1)^{(n-1)|\bullet|+n+|\bullet|'}m_2(\gamma,m_2(\bullet,m_1([X_n,\cdots,X_1])).
\end{align*}
Let $A(\bullet):=(-1)^{(n-1)|\bullet|+1}m_2(\gamma,m_2(\bullet,[X_n,\cdots,X_1])).$ Then $|A|=1$, and
\[ (\ref{inv1})=\str (m_1\circ A + A \circ m_1)+\str(-1)^{(n-1)|\bullet|+n+|\bullet|'}m_2(\gamma,m_2(\bullet,m_1([X_n,\cdots,X_1])),\]
Note that the first term vanishes from Lemma \ref{lem:stp} (3), and the second term lies in the Jacobian ideal
from Lemma \ref{lem:fg} (3). Thus \eqref{inv1} vanishes in Jacobian ring, which proves the claim.
\end{proof}

\begin{definition}\label{defn:qi}
We will say two Lagrangians $L$ and $L'$ are isomorphic if
there exist $p \in CF^0(L,L'), q \in CF^0(L',L)$ such that
$$m_2(p,q) = \be_L + m_1(\gamma_1), m_2(q,p) = \be_{L'} + m_1(\gamma_2)$$
\end{definition}
Note that we have an induced isomorphism
\begin{equation}
\Phi_{q,p}: HF(L,L)  \to HF(L',L'): \alpha \mapsto (-1)^{|\alpha|} m_2(q,m_2(\alpha,p))
\end{equation}
\begin{remark}
The additional sign is natural: indeed, when $p=q=\be_L$ where $\be_L\in CF(L,L)$ is the $\AI$-unit, we have the following identity: 
\[ m_2\big(m_2(\be_L,\alpha),\be_L \big)=m_2(\alpha,\be_L)=(-1)^{|\alpha|}\alpha.\]
\end{remark}

\begin{prop}\label{prop:cm}
We have the following properties
\begin{enumerate}
\item $B(\alpha)$ is invariant under the permutation $\sigma$ of $\{1,\cdots,n\}$.
\item Let $p:L \to L'$ and $q: L' \to L$ be isomorphic via $p,q$ as in Definition \ref{defn:qi}. Then
\[B_L(\alpha)=B_{L'}\big( \Phi_{q,p}(\alpha) \big).\]
\end{enumerate}
\end{prop}
\begin{remark}
\begin{itemize}
 \item $[X_n,\cdots,X_1]$ itself is not invariant under permutation of $\{1,\cdots,n\}$. For example,
in toric cases, the generators form a Clifford algebra under $m_2^b$. \cite{Cho}.
\item We expect that  $B$-invariant only depends on top degree component of $[X_n,\cdots,X_1]$.
This holds for our examples ( see  Corollary \ref{co:degb} for toric cases and Lemma \ref{lem:degb} for $\mathbb{P}^1_{333}$ case.)
In the general case, one should be able to argue as in toric cases, that $B$-invariant can be related to the
composition of open-closed map and closed-open map, and if the output of closed-open map under bulk insertion is a multiple of unit,
then its pairing with $[X_n,\cdots,X_1]$ certainly depends on the top degree part. 
This seem to be the geometric reason behind the claim (1) in the proposition.
\end{itemize}
\end{remark}
\begin{proof}
\begin{enumerate}
\item  For the invariance $B(\alpha)$ under permutation, we use a non-trivial algebraic fact from \cite{PV10}.
Namely, the claim is a corollary of Theorem \ref{perminv} and \ref{mainthm}.
\item By $\AI$-relation, 
\begin{align}
&\str (-1)^{(n-1)|\bullet|}m_2\Big((-1)^{|\alpha|}m_2(m_2(q,\alpha),p),m_2(\bullet,[X_n,\cdots,X_1])\Big) \nonumber \\
=& \str (-1)^{(n-1)|\bullet|+|\alpha|+|m_2(q,\alpha)|'+1}m_2\Big(m_2(q,\alpha),m_2(p,m_2(\bullet,[X_n,\cdots,X_1]))\Big) \label{inv2:eq1} \\
&+ \str (-1)^{(n-1)|\bullet|+|\alpha|+1}m_1\Big(m_3\big(m_2(q,\alpha),p,m_2(\bullet,[X_n,\cdots,X_1])\big)\Big) \nonumber\\
&+ \str (-1)^{(n-1)|\bullet|+|\alpha|+|m_2(q,\alpha)|'+|p|'} m_3\big(m_2(q,\alpha),p,m_2(m_1(\bullet),[X_n,\cdots,X_1]))\big).\nonumber 
\end{align}
Let $A(\bullet):=(-1)^{(n-1)|\bullet|+|\alpha|+1} m_3(m_2(q,\alpha),p,m_2(\bullet,[X_n,\cdots,X_1])).$ Then the second and third terms give $\str (m_1\circ A + A\circ m_1)=0.$ Applying $\AI$-relations and graded symmetry of supertraces repeatedly as above, we ultimately have
\begin{align*}
(\ref{inv2:eq1})=&\str(-1)^{(n-1)|\bullet|+|\alpha|+n}m_2\Big(q,m_2\big(\alpha,m_2(m_2(p,\bullet),[X_n,\cdots,X_1])\big)\Big) \\
=& \str(-1)^{(n-1)|\bullet|}m_2\big(q,m_2\big(\alpha,m_2(\bullet,[X_n,\cdots,X_1])\big)\big) \circ m_2(p,\bullet) \\
=& \str(-1)^{(n-1)|\bullet|}m_2(p,\bullet)\circ m_2\big(q,m_2\big(\alpha,m_2(\bullet,[X_n,\cdots,X_1])\big)\big)\\
=& \str(-1)^{(n-1)|\bullet|}m_2\Big(p,m_2\big(q,m_2(\alpha,m_2(\bullet,[X_n,\cdots,X_1]))\big)\Big)\\
=& \str(-1)^{(n-1)|\bullet|+|p|'+1}m_2\big(m_2(p,q),m_2(\alpha,m_2(\bullet,[X_n,\cdots,X_1]))\big)\\
&+ \str(-1)^{(n-1)|\bullet|+|p|'+1}m_1 m_3\big(p,q,m_2(\alpha,m_2(\bullet,[X_n,\cdots,X_1]))\big) \\
&+ \str(-1)^{(n-1)|\bullet|+|q|'+1}m_3\big(p,q,m_1 m_2(\alpha,m_2(\bullet,[X_n,\cdots,X_1]))\big).
\end{align*}
Last two terms vanish again due to the graded symmetry, so 
\begin{align*} 
(\ref{inv2:eq1})=&\str(-1)^{(n-1)|\bullet|+|p|'+1}m_2(m_2(p,q),m_2(\alpha,m_2(\bullet,[X_n,\cdots,X_1])))\\
=&\str(-1)^{(n-1)|\bullet|}m_2({\be}_L,m_2(\alpha,m_2(\bullet,[X_n,\cdots,X_1]))) \\ 
&+ \str(-1)^{(n-1)|\bullet|}m_2(m_1(\gamma_1),m_2(\alpha,m_2(\bullet,[X_n,\cdots,X_1])))\\
=&\str(-1)^{(n-1)|\bullet|}m_2(\alpha,m_2(\bullet,[X_n,\cdots,X_1])).
\end{align*}
\end{enumerate}
For the last equlity, we can handle the term $m_2(m_1(\gamma_1),m_2(\alpha,m_2(\bullet,[X_n,\cdots,X_1])))$ as in the proof of Theorem \ref{thm:bw} and we omit it.
\end{proof}

Now, we give a proof of Lemma \ref{lem:fg}
\begin{proof}
Let us prove (1).
For the immersed case, take a derivative of $m(e^b) = W(b) \cdot \be$. As $b = \sum x_i X_i$,  we have
$$\frac{\partial W}{\partial x_i}  \cdot \be = \frac{\partial m(e^b)}{\partial x_i} = m_1^b(X_i).$$
This proves the first claim. 
For the toric cases, recall the following Lemma 11.8 from \cite{FOOOtoric1}
\begin{lemma}\label{lem:cpmk}
Let $\gamma$ be a harmonic one form on a Lagrangian torus fiber $L$.
$$m_{k,\beta}(\gamma,\cdots,\gamma) = \frac{c_\beta}{k!}(\partial \beta \cap \gamma)^k \be$$
\end{lemma}
Using the similar method, one can prove the following.
\begin{lemma}
Let $\gamma_1,\cdots,\gamma_n$ be 1-forms on $L$.
$$ \sum_{\sigma \in S_k} m_{k,\beta}(\gamma_{\sigma(1)},\cdots,\gamma_{\sigma(k)}) = c_\beta \prod_{i=1}^k (\partial \beta \cap \gamma_i) $$
\end{lemma}
In particular, this implies that
$$m_{1,\beta}^b(E_i^*)= ( \partial \beta \cap E_i^* ) \cdot m_{0,\beta}^b, m_{2,\beta}^b(E_i^*, E_j^*) = (\partial \beta \cap E_i^*) \cdot (\partial \beta \cap E_i^*) m_{0,\beta}^b$$
Thus we have
\begin{equation}\label{eq:dd}
y_i \frac{\partial W}{\partial y_i} \cdot \be = \frac{\partial W}{\partial x_i}  \cdot \be = m_{1}^b(E_i^*)
\end{equation}
\begin{equation}\label{eq:dd1}
y_i y_j \frac{\partial^2 W}{\partial y_i \partial y_j} \cdot \be=  \frac{\partial^2 W}{\partial x_i \partial x_j} \cdot \be = m_2^b(E_i^*, E_j^*) +m_2^b(E_j^*, E_i^*)
\end{equation}

For (2), note that $m_2 = m_{2,0} + \sum_\beta m_{2,\beta}$
and $m_{2,0}^b = m_{2,0}$ for both immersed and toric cases. Since $m_{2,\beta}$ has
higher valuation, and the $\AI$-algebras involved are gapped filtered, the result follows.

For (3), consider the Leibniz rule 
$$m_1^b(m_2^b(\alpha,X_i))= - m_2^b(m_1^b(\alpha),X_i)  + (-1)^{|\alpha|} m_2^b(\alpha,m_1^b(X_i)).$$
It shows that if $m_1^b(\alpha)$ lies in the Jacobian ideal, then $m_1^b(m_2^b(\alpha,X_i))$ also lies in the Jacobian ideal.
This proves (3).
\end{proof}

\section{Boundary-bulk map and $B$-invariants}\label{sec:bb}
In this section, we  show that diagram \eqref{diagram0} commutes. This explains the relation between  $B$-invariants and boundary-bulk maps in $B$-model.
\begin{definition} 
 The {\em boundary-bulk map}
 $$\tau: End_{\AI(MF(W_\bL)) }(P, Q ) \to Jac(W_\bL) \cdot dx_1 \wedge \cdots \wedge dx_n$$ is defined as
$$f \mapsto  (-1)^{n} \frac{1}{n!}\str( f \circ (d Q)^{\wedge n}).$$
\end{definition}



Localized mirror functor allows us to interpret $Q$ as a (deformed) Floer differential $(-m_1^{\xi,b})$, and $\partial_{x_i}Q$ as a (deformed) $m_2$-products with $X_i$ (i.e. $m_2^{\xi,b}(\bullet,X_i)$).
Thus, the boundary-bulk map provides a geometric configuration of multi-crescent discs in Figure \ref{fig:supertrace}.
We prove in the following main theorem that this can be reduced to bi-crescent discs with $m_2$-discs attached as in Figure \ref{fig:supertrace2} using $\AI$-identities and the properties of supertraces
\begin{theorem}\label{mainthm}
We have following identity of supertraces (modulo Jacobian ideal relations).
	\begin{align} 
	\str\Big( \mathcal{F}^\bL(\alpha) \circ \partial_{x_1} m_1^{\xi,b} \circ \cdots \circ \partial_{x_n} m_1^{\xi,b}  \Big) \label{supertrace}= \str\big( \bullet \mapsto (-1)^\epsilon  \mathcal{F}^\bL(\alpha) \circ m_2^{\xi,b,b}(\bullet,[X_n,...,X_1]) \big),
	\end{align}
	where $[X_n,\cdots, X_1]$ is defined in Definition \ref{defn:fv} and 
		\[ \epsilon = 1 + \frac{n(n+1)}{2} + (n-1)|\bullet|. \]
\end{theorem}
\begin{figure}
	\includegraphics{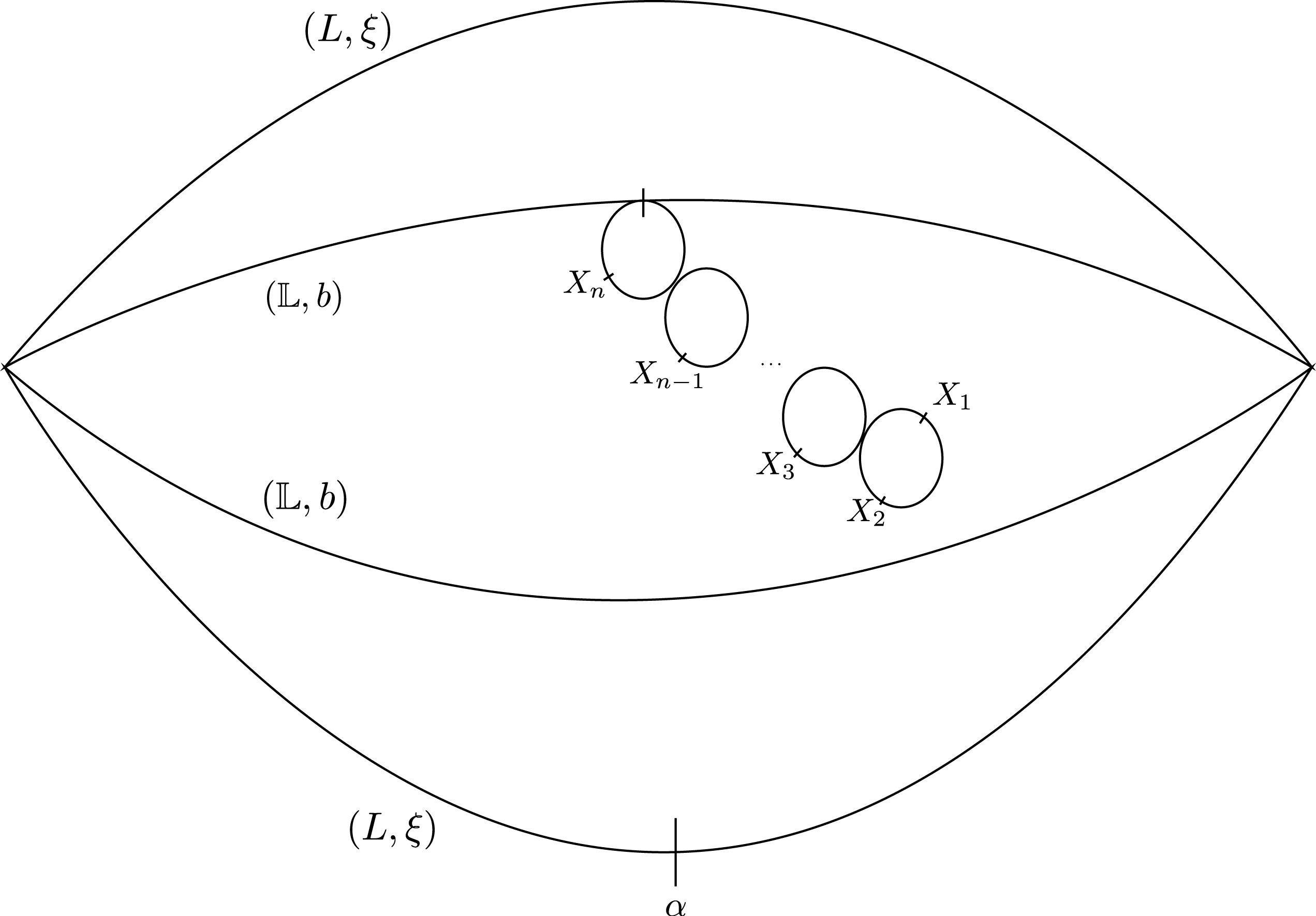}
	\caption{RHS of \eqref{supertrace}.}
	\label{fig:supertrace2}
\end{figure}
For the toric case, the above can be stated with $y_i$ variables as well, since we have  
\[\frac{\partial}{\partial x_i}= y_i \frac{\partial}{\partial y_i}, dx_i = \frac{dy_i}{y_i}.\]
Hence
 $$\str(f\circ \partial_{x_1} Q \circ \cdots \circ \partial_{x_{n}} Q)\cdot dx_1 \wedge \cdots \wedge dx_n =  \str(f\circ y_1\partial_{y_1} Q \circ \cdots \circ y_n \partial_{y_{n}} Q)\cdot \frac{dy_1}{y_1} \wedge \cdots \wedge \frac{dy_n}{y_n}$$
$$ = \str(f\circ \partial_{y_1} Q \circ \cdots \circ \partial_{y_{n}} Q)\cdot dy_1 \wedge \cdots \wedge dy_n$$

To relate this to the commutativity, we recall from Polishchuk-Vaintrob:
\begin{theorem}[Corollary 3.2.4 \cite{PV10}]\label{perminv}
We have 
	\begin{equation} 
	\frac{1}{n!}\str(f\circ (dQ)^{\wedge n} ) = \str(f\circ \partial_{x_1} Q \circ \cdots \circ \partial_{x_{n}} Q)\cdot dx_1 \wedge \cdots \wedge dx_n.
	\end{equation}
	Moreover the right hand side of the above equation is invariant  under the permutation of indices $(1,2,\ldots,n)$
	( in $Jac(W_\bL)dx_1 \cdots dx_n$ ).
\end{theorem}
Combining the above two theorems,  we obtain the commutativity of the diagram \ref{diagram7} as a corollary.

\begin{corollary}\label{cor:ll}
\begin{equation}\label{eq:ll}
  (-1)^{n} \frac{1}{n!}\str(  \mathcal{F}^\bL(\alpha) \circ (dQ)^{\wedge n} ) =   B(\alpha)
  \end{equation}
\end{corollary}
\begin{proof}
Let us discuss the immersed case first.
Under localized mirror functor, we have $Q = - m_1^{b_0,b}$ (which gives additional $(-1)^n$ from $(dQ)^{\wedge n} = d(-m_1^{\xi,b})^{\wedge n}$).
Thus, we can combine the above two theorem to show that LHS of \eqref{eq:ll} equals
$$ \str\big( \bullet \mapsto (-1)^{\epsilon'}  \mathcal{F}^\bL(\alpha) \circ m_2^{b_0,b,b}(\bullet,[X_n,...,X_1]) \big)dx_1\wedge\cdots\wedge dx_n $$
where $\epsilon' =\frac{n(n+1)}{2} +1 + (n-1)|\bullet|$, which is the sign for $B$-invariant.

\end{proof}

Let us prove Theorem \ref{mainthm}.
\begin{proof}
	Let $\bullet \in CF^*(L,\LL).$ Since $m_1^{\xi,b}(\bullet)=\sum_{k\geq 0} m_{k+1}(\xi,...,\xi,\bullet,b,...,b)$, we have \[\partial_{x_i}(m_1^{\xi,b}(\bullet))=\sum_{k\geq 1} m_{k+1}(\xi,...,\xi,\bullet,b,...,b,X_i,b,...,b)=m_2^{\xi,b,b}(\bullet,X_i).\] 
	Hence, we can write 
	\begin{equation}\label{ppp1}
	\partial_{x_1}(m_1^{\xi,b})\circ \cdots \circ \partial_{x_n}(m_1^{\xi,b})(\bullet)
	=m_2^{\xi,b,b}\Big(m_2^{\xi,b,b}\big(\cdots m_2^{\xi,b,b}(m_2^{\xi,b,b}(\bullet,X_n),X_{n-1}),\cdots,X_2\big),X_1\Big).
	\end{equation}
	
	We will modify the above expression using various $\AI$-equations.
	
	Let us denote 
	$$K_j(\bullet):= m_2^{\xi,b,b}\big(m_2^{\xi,b,b}(\cdots m_2^{\xi,b,b}(\bullet,X_n),X_{n-1}),\cdots,X_j\big),$$
	for $ 1\leq j \leq n$. The degree of $K_j(\bullet)$ is $|\bullet| + n-j+1$.
	
	The expression \eqref{ppp1} can be written as $$m_2^{\xi,b,b}\big(m_2^{\xi,b,b}(K_3(\bullet),X_2),X_1\big).$$
	We consider the following $\AI$-equation of $K_3(\bullet)$, $X_2$ and $X_1$, which contains the above term.
	
	\begin{align}
  0= m_2^{\xi,b,b}\Big(m_2^{\xi,b,b}\big(K_3(\bullet),X_2\big),X_1\Big) & +(-1)^{\epsilon_1} m_2^{\xi,b,b}\Big(K_3(\bullet),m_2^{b,b,b}(X_2,X_1)\Big) \label{contribute} \\
  +m_1^{\xi,b}m_3^{\xi,b,b,b}\Big(K_3(\bullet),X_2,X_1\Big) 
	& +m_3^{\xi,b,b,b}\Big(m_1^{\xi,b}\big(K_3(\bullet)\big),X_2,X_1\Big) \label{htpy1}\\
 + (-1)^{\epsilon_2} m_3^{\xi,b,b,b}\Big(K_3(\bullet),m_1^{b,b}(X_2),X_1\Big)  
	& +(-1)^{\epsilon_3}  m_3^{\xi,b,b,b}\Big(K_3(\bullet),X_2,m_1^{b,b}(X_1)\Big). \label{unit1}
	\end{align}
	Here, the terms  containing $m_0^b$ or $m_0^{\xi}$ vanishes due to the weakly unobstructedness.
	Here $\epsilon_1=n-3 + |\bullet| \equiv 1 + n + |\bullet|$, which is the shifted degree of $K_3(\bullet)$.
    Also two terms in \eqref{unit1} lies in the Jacobian ideal from Lemma \ref{lem:fg} (3).
    
      Now, we want to argue that two terms in \eqref{htpy1} can be further transformed into a nice form.
Let us denote
 $$H(\bullet):= m_3^{\xi,b,b,b}\Big(K_3(\bullet),X_2,X_1\Big).$$
\begin{lemma}
	The expression \eqref{htpy1} equals 
	$$\delta(H(\bullet))  = m_1^{\xi,b} H(\bullet) - (-1)^{|H|} H( m_1^{\xi,b}(\bullet) )$$
	modulo Jacobian ideal.
\end{lemma} 
\begin{proof}
 Note that from the $\AI$-equation, we have
 $$ m_1^{\xi,b}(K_j(\bullet)) = - m_2^{\xi,b,b}\big(m_1^{\xi,b}(K_{j+1}(\bullet)),X_j\big) 
  \pm m_2^{\xi,b,b}(K_{j+1}(\bullet),m_1^{b,b}(X_j)) $$
 Note that the second term on the right hand side lies in the Jacobian ideal.
 Hence, successively applying the above, we have (modulo Jacobian ideal)
  $$ m_3^{\xi,b,b,b}\Big(m_1^{\xi,b}\big(K_3(\bullet)\big),X_2,X_1\Big) 
  = m_3^{\xi,b,b,b}\Big( (-1)m_2^{\xi,b,b}\big(m_1^{\xi,b}(K_{4}(\bullet)),X_3\big),X_2,X_1)$$
  $$ = \cdots = (-1)^{n-2} m_3^{\xi,b,b,b}\Big( m_2^{\xi,b,b}\big( \cdots  m_2^{\xi,b,b}(m_1^{\xi,b}(\bullet),X_n),
  \cdots,X_3\big),X_2,X_1 \Big)$$
  $$ = (-1)^{n-2} m_3^{\xi,b,b,b}\Big( K_3(m_1^{\xi,b}(\bullet)),X_2,X_1) = -(-1)^{|H|}H(m_1^{\xi,b}(\bullet)).$$
 This proves the lemma.
\end{proof}

This lemma is helpful because of Lemma \ref{strprop}(3).
Namely, if $\delta (p_\alpha) =0$, then
$$Str(p_\alpha \delta(H)) = \pm Str(\delta(p_\alpha)H) = 0.$$
Therefore, the contribution of the term \eqref{htpy1} to the supertrace is zero.

From \eqref{contribute}, the contribution to the supertrace of
$m_2^{\xi,b,b}\Big(m_2^{\xi,b,b}\big(K_3(\bullet),X_2\big),X_1\Big)$
is the same as that of $ (-1)^{n + |\bullet|} m_2^{\xi,b,b}\Big(K_3(\bullet),m_2^b(X_2,X_1)\Big)$.
Now, the latter can be written as 
$$ (-1)^{n + |\bullet|}  m_2^{\xi,b,b}\Big(m_2^{\xi,b,b}\big( K_4(\bullet),X_3),m_2^b(X_2,X_1)\Big).$$
By repeating the same arguments, we find that the contribution of the original term equals
that of
$$ (-1)^{\epsilon} m_2^{\xi,b,b}\Big(\bullet, m_2^b(X_n,m_2^b(X_{n-1},\cdots m_2^b(X_3,m_2^b(X_2,X_1))\cdots))\big) = m_2^{\xi,b,b}\Big(\bullet,  [X_n,...,X_1] \Big).$$
Here 
\[\epsilon =  (n-2) + |\bullet| + (n-3) + |\bullet| + \cdots + 0 + |\bullet| \equiv 1 + \frac{n(n+1)}{2} + (n-1)|\bullet| \mod 2.\]
\end{proof}

\section{Morse-Bott version of $B$-invariants for toric manifolds}
When we defined $B: HF(L,L) \to Jac(W_\bL)$ in section \ref{sec:B}, we excluded the
case $L = \bL$ for toric type. Namely, if the reference Lagrangian $\bL$ is a torus, then
we needed to introduce the notion of hyper-tori's to construct localized mirror functors in Section \ref{sec:toriclm}
and the localized mirror functor $\mathcal{F}^{\bL}$ on the object $\bL$ itself was ill-defined
and had to work with its hamiltonian perturbation.

But we show that  $B$-invariant  is well defined for the Bott-Morse case for toric manifolds.
\begin{prop}\label{prop:bm}
Let $M$ be a compact toric manifold, and $L$ a Lagrangian torus fiber. $B$-invariant for the Morse-Bott case $\bL=L$
$$B_{L}: HF(L,L) \to Jac(W_{L}) \frac{dy_1}{y_1} \wedge \cdots \wedge \frac{dy_n}{y_n} $$
is well-defined .
For any $L_1 = \phi_H(L)$, which is quasi-isomorphic  to $L$ in the sense of \ref{defn:qi},
their $B$-invariants $B_{L}$ and $B_{L_1}$ can be canonically identified. 
\end{prop}
The localized mirror functor $\mathcal{F}^{\bL}$ was ill-defined for $\bL=L$ due to the choice of mirror variable $y_i$ instead of $x_i$ when $y_i=e^{x_i}$. While mirror functor on $L$ may not be written in $y_i$-variables, we will see that $B$-invariant is written in $y_i$-variables. 
On the other hand, we still need mirror functor to relate $B$-invariant with the boundary-bulk map of $B$-model (and thus to Kapustin-Li pairing).  The second claim in the above Proposition \ref{prop:bm} will provide this compatibility, which was already proved for the transverse case $B_{L_1}$ in the previous section.


%
%
%
%

We first explain more detailed construction of  localized mirror functor and $B$-invariants for toric manifolds.
\subsection{Floer theory for general toric manifold}\label{sec:cmp}
Recall that Lagrangian Floer theory for toric manifolds has been developed in  \cite{CO} (Fano case), \cite{FOOOtoric1} (general case) and we will follow the latter version. We will consider cases without bulk-deformations for simplicity.
 
 Let $\pi:M \to \mathbb{R}^n$ be the moment map with a polytope $P$. For $u \in int(P)$, $L(u) = \pi^{-1}(u)$ is a Lagrangian torus fiber, which is a $T^n$-orbit. Holomorphic discs with boundary on $L(u)$ has been classified in \cite{CO}.
 Note that $T^n$-action on $M$ induces an action on the moduli space of (stable) holomorphic discs with boundary on $L(u)$.
 Fukaya-Oh-Ohta-Ono considered $T^n$-equivariant Kuranishi structure on this moduli space to show that a torus fiber $L(u)$ on a general toric manifold $M$ is weakly unobstructed.
Here $T^n$-equivariance is used in the following way.  Any moduli space of holomorphic discs of Maslov index $\leq 0$ should vanish as the expected dimension of the moduli space is less than $n=\dim(L)$ (for example, 
$\dim (\mathcal{M}_1(\beta)) = n+\mu(\beta) +1-3$), but $T^n$-equivariance implies that any non-trivial moduli space should have dimension at least greater than or equal to $n$.

 Each torus fiber $L(u)$ is associated with its (canonical) unital filtered $\AI$-algebra 
$$\big( H(L(u),\Lambda_0), \{m_k\}, PD[L(u)]\big).$$
  Elements  $b \in H^1(L(u),\Lambda_0)$ are shown to be weak bounding cochains, and $m(e^{b}) = W(b) \be_{L(u)}$ defines the potential function.
 Namely, we have $$W^{FOOO}: \cup_{u \in int(P)} H^1(L(u),\Lambda_0) \to \Lambda_0.$$
 One can fix $H^1(L(u),\Lambda_0) \cong (\Lambda_0)^n$ and regard $W^{FOOO}$ as a function
 $$W^{FOOO}(x_1,\cdots,x_n,u_1,\cdots,u_n): (\Lambda_0)^n \times \textrm{Int}(P) \to \Lambda_0$$ 
By setting $y_i = e^{x_i}$, potential function 
$W^{FOOO}$ can be written as a sum
$$W^{FOOO}(x_1,\cdots,x_n,u_1,\cdots,u_n) = \sum T^{c_i(u)} P_i(y_1,\cdots,y_n)$$
for $P_i$'s are Laurant polynomials which do not depend on $u$. If $M$ is toric Fano, then each $P_i$ corresponds to
toric divisor $D_i$, and $c_i(u)$ is the area of basic Maslov index two disc which intersects the $D_i$ once, and do not intersect other $D_j$ for $j \neq i$. In general it could be an infinite sum.
For a fixed $u \in \textrm{Int} (P)$, we may set $W^{u,FOOO}(y_1,\cdots,y_n) = W^{FOOO}(x_1,\cdots,x_n,u_1,\cdots,u_n)$.
\begin{lemma}
Let $W_{L(u)}(x)$ be the potential function of the torus fiber $L(u)$ defined in Section \ref{sec:toriclm}.
Then we have $$W_{L(u)}(y_1,\cdots,y_n) = W^{u,FOOO}(y_1,\cdots,y_n)$$
\end{lemma}
\begin{proof}
Recall that in Section  \ref{sec:toriclm}, we used holonomy variables in $\C^*$ (or in $\Lambda_0^*$) to define localized mirror functor, but the variables of RHS are  from the bounding cochains. Using
Lemma \ref{lem:cpmk}, one can turn degree one insertions into exponential form to identify these two potential functions.
\end{proof}
For the rest of the paper, we identify these two potentials and write it as $W$ for simplicity.
Recall that in \cite{FOOOtoric3}, authors introduce $y_i(u) = T^{-u_j}y_i$ for $i=1,\cdots, n$. This defines the valuation
$\sigma_T^u$ on $\Lambda[y(u), y(u)^{-1}]$ such that
$$\sigma_T^u(y_i(u))=0, \sigma_T^u(T) =1.$$
Define $\sigma_T^P= \inf_{u \in P} \sigma_T^u$.
The completion of  $\Lambda_0[y(u), y(u)^{-1}]$ with respect to $\sigma_T^P$ (resp. $\sigma_T^u$) is denoted as
$\Lambda \ll y,y^{-1} \gg^P_0$ (resp. $\Lambda \ll y,y^{-1} \gg^u_0$).
Replacing $\Lambda_0$ by $\Lambda$, we can define
$\Lambda \ll y,y^{-1} \gg^P$ (resp. $\Lambda \ll y,y^{-1} \gg^u$) in the same way.

For Fano case, Jacobian ring can be defined on $\Lambda[y(u), y(u)^{-1}]$, but in general,
Fukaya-Oh-Ohta-Ono uses the following definition of Jacobian ring
\begin{definition}
$$Jac(W) = \frac{ \Lambda \ll y,y^{-1} \gg^P_0}{ Clos_{\sigma_T^P} \big( \{y_i \frac{\partial W}{\partial y_i} \mid i=1,\cdots,n \})}$$
where the closure is taken with respect to the norm $e^{- \sigma_T^P}$.
\end{definition}

On the other hand, in localized mirror formalism, we fix a torus fiber $L(u)$ and use a single valuation $\sigma_T^u$.
Note that there is a canonical map  (for any $u \in P$)
$$ \frac{ \Lambda \ll y,y^{-1} \gg^P_0}{ Clos_{\sigma_T^P} \big( \{y_i \frac{\partial W}{\partial y_i} \mid i=1,\cdots,n \})}
\to \frac{ \Lambda \ll y,y^{-1} \gg^u_0}{ Clos_{\sigma_T^u} \big( \{y_i \frac{\partial W}{\partial y_i} \mid i=1,\cdots,n \})}$$
and the latter is the Jacobian ring that is used to define $B$-invariant. We will use this map to compare them.

Let us recall construction of cyclic symmetric $\AI$-algebras given in \cite{FOOOtoric3}(which is based on \cite{Fu}).
The perturbation on the moduli space of $J$-holomorphic discs should be invariant under the cyclic permutation of boundary marked points, and for this the technique of continuous family of multi-sections is used.

We will leave the details to the above references, but only remind an important issue regarding unobstructedness.
$T^n$-equivariant perturbation in \cite{FOOOtoric1} was used to show that $H^1(L(u),\Lambda_0)$ form
weak bounding cochains. But when using the technique of continuous family of multi-sections, this does not hold.
because the moduli space of negative expected dimension may be non-empty after perturbation.
To find weak bounding cochains, one should use the $\AI$-quasi-isomorphism between these two constructions,
and transfer weak bounding
cochains from one to the other.

Let $(H,\{m_{k,\beta}^{T^n,\rho} \})$ (resp.   $(H,\{m_{k,\beta}^{c,\rho} \})$) be the $\AI$-algebra of $L_0$ with $T^n$-equivariant (resp. cyclic symmetric $T^n$-equivariant) perturbation in \cite{FOOOtoric1} (resp. in \cite{FOOOtoric3}).
Let $b= b_0 + b_+$, where $b_0$ represents the holonomy of flat $\C$-bundle $\rho$
(recall that we have $\rho:H_1(L,\Z) \to \C^*$ by $\rho(\gamma) = e^{\gamma \cap b_0}$), and $b_+$ is a weak bounding cochain in $H^1(L,\Lambda_+)$ of $(H,\{m_{k,\beta}^{T^n,\rho} \})$. 
 We may also denote the bounding cochain for the latter by $b^c$, which may or may not be of degree one.
Corollary 3.2.22\cite{FOOOtoric3} provides an isomorphism $\Phi: (H,\{m_{k,\beta}^{T^n,\rho} \}) \to (H,\{m_{k,\beta}^{c,\rho} \})$, which sends weak Maurer-Cartan elements $b$ to weak Maurer-Cartan elements $b^c = b_0 +  \Phi_*(b_+)$. Here 
$$\Phi_*(b_+) = \sum_{k=0}^\infty \Phi_k(b_+,\cdots,b_+)$$
and we remark that even when $b_+ =0$, $\Phi_*(b_+)$ could be non-trivial due to $\Phi_0(1)$,
the part of the $\AI$-functor with zero input and one output. We refer readers to \cite{FOOOtoric3} for full details.

For localized mirror functor, we pick and fix a Lagrangian torus fiber $L(u)$ equipped with
bounding cochain $\xi = \xi_0+\xi_+$ that has non-trivial Floer cohomology. Let us denote the potential value $W(\xi) = \lambda$.
Recall from Section \ref{sec:toriclm} that for toric cases, we use holonomy variables $y_1, \cdots, y_n$( with 
$y_i = e^{x_i}$, $b_0 = \sum x_i \theta_i$). We take $b = b_0 + \xi_+$, so that  the potential function has
a critical point $e^{b_0} = e^{\xi_0}$. Now, we may use either one of perturbation scheme to define localized mirror functors, which should result in quasi-isomorphic functors.

Consider a Fukaya category of $M$, $\mathcal{F}u(M)_\lambda$ whose object is a  weakly unobstructed Lagrangian submanifold $L$ in $M$ (with the same potential value $\lambda$ as $(L(u), \xi)$) such that $L$ transversely intersect the $\bL=L(u)$. Note that we do not include $L(u)$ as an object of $\mathcal{F}u_\lambda$.
We also consider a Fukaya category of $M$, $\mathcal{F}u^+_\lambda$, which is obtained by adding the object $L$ to $\mathcal{F}u_\lambda$.
We use $\AI$-operations of $\mathcal{F}u^+_\lambda$ to define an $\AI$-functor from the Fukaya category $\mathcal{F}u_\lambda$ to the matrix factorization category of $W_{L(u)}$. This defines a localized mirror functor. See \cite{CHL2} for more details.

\subsection{Proof of Proposition \ref{prop:bm}}
The algebraic definition of $B$-invariant works even in the Bott-Morse case, but 
the issue is whether the supertrace lies in Jacobian ring. We need to show that the holonomy we get is a Laurant monomial on $y_i$, not just an expression of $x_i$. 
If holonomy is evaluated along a loop (as in $m_k^{b,\cdots,b}$), then we get the former, if along an arc (as in $m_k^{0,\cdots,0,b}$), we get the latter.
For $B$-invariant, although we use $m_2^{0,0,b}, m_2^{0,b,b}$, but we take a supertrace.
In a transversal case, it is clear that trace is non-trivial only
if the $L_0$-part of the boundaries of the pseudo-holomorphic polygons connect to a loop. 
For the Bott-Morse version, this argument does not work.

To prove the first claim of Proposition \ref{prop:bm}, we find an equivalent definition of $B$-invariant whether
the holonomy contribution is apparently on $y_i$-variable.
But we can follow Fukaya-Oh-Ohta-Ono to use   the moduli space of pseudo-holomorphic annuli to construct a cobordism from the configuration
for $B$-invariant to another configuration as in Figure \ref{fig:cobordism}. Then it is quite clear that
the holonomy contribution for the latter is on a loop, and thus given as a Laurant monomial on $y_i$.

More explicitly, we proceed as follows.
We may first rewrite $B$-invariant using  inner product.
Let us denote the basis of  $CF(L,L)$ as $\{e_I\}$, and set $g_{IJ} = \langle e_I, e_J \rangle_{PD}$, and denote by $g^{IJ}$ its inverse matrix.  As we consider deformed Lagrangian $(L,\xi)$, $B$-invariant can be written as
\begin{eqnarray*}
 B(\alpha) &=& (-1)^{\frac{n(n+1)}{2}+1} \str \big( \bullet \mapsto (-1)^{(n-1)|\bullet|}m_2^{\xi,\xi,b}(\alpha,m_2^{\xi, b,b} (\bullet,[X_n,\cdots,X_1])) \big)dx_1 \cdots dx_n \\
  &=& (-1)^{\frac{n(n+1)}{2}+1} \tr \big( \bullet \mapsto  (-1)^{n|\bullet|} m_2^{\xi,\xi,b}(\alpha,m_2^{\xi,b,b}(\bullet,[X_n,\cdots,X_1])) \big)dx_1 \cdots dx_n \\
  &=& (-1)^{\frac{n(n+1)}{2}+1} \sum_{I,J}  g^{JI} \langle m_2^{b,\xi,\xi}(e_J, \alpha),m_2^{\xi,b,b}(e_I,[X_n,\cdots,X_1]))\rangle_{PD} \;\; dx_1 \cdots dx_n
 \end{eqnarray*}
We use the fact that
 $$\str\big(\bullet \mapsto \phi(\bullet)\big) = \mathrm{Tr} \big(\bullet \mapsto (-1)^{|\bullet|} \phi(\bullet)\big).$$
For the last equality, we use Proposition \ref{ksoc}.

Let us recall a few notations first. Open-closed map and closed open map in toric setting was denoted as
\begin{eqnarray*}
 i_{*,qm,b^c} &:&  HF((L(u),b^c), (L(u),b^c)) \to H(M,\Lambda_0) \\
i^{c*}_{qm,b^c} &:& H(M,\Lambda_0) \to  HF((L(u),b^c), (L(u),b^c))
\end{eqnarray*}
In Section 3.3 in \cite{FOOOtoric3} $ i_{*,qm,b^c}$ is constructed  using $p$-perturbation,
and $i^{c*}_{qm,b^c}$ is constructed using $q^c$-perturbation. In both cases  the compatible disc moduli space should have cyclic symmetry. In FOOO,  $i_{*,qm,b^c}$ is sometimes denoted as  $i_{*,qm,b,u}$.

In \cite{FOOOtoric3}, the following theorem has been proved, by analyzing the moduli space of holomorphic annuli.
\begin{theorem}\cite[Theorem 3.4.1]{FOOOtoric3}\label{fooothm1}
We have (when $|v|=|w|=n$)
$$ \langle i_{*,qm,b^c}(v) ), i_{*,qm,b^c}(w) \rangle_{PD_M}= (-1)^{\frac{n(n-1)}{2}} \sum_{I,J}g^{IJ} \langle m_2^{b^c,b^c,b^c}(e_I, v), m_2^{b^c,b^c,b^c}(e_J, w) \rangle_{PD_{L(u)}}$$
\end{theorem}

By just replacing boundary insertions of weak bounding cochains, we get
\begin{equation}\label{eq10}
 \langle i_{*,qm,\xi^c}(v) ), i_{*,qm,b^c}(w) \rangle_{PD_M}= (-1)^{\frac{n(n-1)}{2}} \sum_{I,J} g^{IJ} \langle m_2^{b^c,\xi^c,\xi^c}(e_I, v), m_2^{\xi^c,b^c,b^c}(e_I, w) \rangle_{PD_{L(u)}}
\end{equation}
Note that the holonomy contribution for open-closed maps $i_{*,qm,b^c}$ are manifestly measured along loops, and hence written in Laurant monomials of $y_i$'s. This proves Proposition \ref{prop:bm} (1)

For (2), we compare  Bott-Morse and transversal $B$-invariants using Proposition \ref{prop:cm} (2).
Namely, for a Lagrangian torus fiber $L$, we consider a Hamiltonian isotopy $\phi_H$ to deform it to
 $L_1 = \phi_H(L)$. We can equip $L_1$ with the weak bounding cochains $b_1^c$ corresponding
 to $b^c$ of $L$.    
 We may require that $\phi_H$ is very close to identity, and $\phi_H(L) \cap L$ intersect
 at $2^n$-points. Since $HF((L,b^c),(L,b^c)) \cong HF((L,b^c), (\phi_H(L), b_1^c))$
 , Floer differential for $CF((L,b^c), (\phi_H(L),b_1^c))$ should vanish.
 We take a generator corresponding to the unit in  $HF((L,b^c),(L,b^c))$ to be $p \in CF((L,b^c), (\phi_H(L),b_1^c)),  q \in CF((\phi_H(L),b_1^c),(L,b^c))$. These $p,q$ provides the required isomorphism in the sense of Definition \ref{defn:qi}.

Let us also recall the following fact which has an important corollary.
\begin{theorem}\cite[Theorem 3.3.8, Corollary 3.3.9]{FOOOtoric3}\label{thm:vd}
For $Q \in \mathcal{A}(X), h \in H^*(L(u),\Lambda_0)$, we have
$$\langle i^{c*}_{qm,b^c}(Q), h \rangle_{PD_{L(u)}} = \langle Q, i_{*,qm,b^c}(h) \rangle_{PD(M)}$$
Moreover, if degree $n$-component of $h$ is zero, then the RHS vanishes.
\end{theorem}
Combining this  theorems with \eqref{eq10}, we obtain
\begin{corollary}\label{co:degb}
In the toric  Bott-Morse case,
$B(\alpha)$ only depends on degree $n$ component of  $[X_n,\cdots,X_1]$
\end{corollary}

\section{Kodaira-Spencer map and $Z$-invariant}
In this section, we will explain the relation of $B$-invariant and Kodaira-Spencer map of Fukaya-Oh-Ohta-Ono for toric manifolds. In particular, we will prove a commutative diagram \eqref{diagram1}, from which we can find an interesting scaling factor $c_L$ of the isomorphism $I$

\begin{definition}\label{def:fv}
The ratio of Floer volume class and classical volume class of Lagrangian $L$ is denoted as  
 $c_L = vol^{Floer} / vol_L$. i.e. from Definition \ref{defn:fv}, we have
$$c_L \cdot vol_L =  (-1)^{n+1} [X_n, \cdots,X_1]^{(n)}.$$
\end{definition}

\subsection{Commutative diagram}
We prove the commutativity of the following diagram.
\begin{theorem}\label{thm:kscommute}
Let $M$ be a compact toric manifold, and $L$ be a Lagrangian torus fiber.
Define a map $I : Jac(W_L) \to Jac(W_L)  \frac{dy_1}{y_1} \cdots \frac{dy_n}{y_n}$ to be
$$ h \mapsto    h \cdot (c_L  \frac{dy_1}{y_1} \cdots \frac{dy_n}{y_n}). $$
Then, we have a commuting diagram
\begin{equation}\label{diagram1}   \xymatrix{
	HF(L,L)  \ar[rr]^-{OC} \ar[ddrr]^-{B} &  \;\;\;\;   &  QH^*(M) \ar[d]^{\ks}  \\
	& & Jac(W_L) \ar[d]^-{I} \\
	 &   \;\;\;\;   &   \;\;\;\;    Jac(W_L) \frac{dy_1}{y_1} \cdots \frac{dy_n}{y_n}}
\end{equation}
i.e. we have 
$$ c_L \cdot ( \ks \circ OC (\alpha) )  \frac{dy_1}{y_1} \cdots \frac{dy_n}{y_n} =  B(\alpha)$$
\end{theorem}
\begin{remark}
For $\ks$, we have fixed a reference Lagrangian $\bL=L$. So $\ks$ should be denoted as $\ks^L$, but we just write $\ks$ for simplicity.
\end{remark}
\begin{remark}
We restrict to the case that $M$ is a general compact toric manifold to use the constructions of Fukaya-Oh-Ohta-Ono. 
But the arguments in this section should generalize to other cases in principle.
For example, for the case of orbi-sphere $\mathbb{P}^1_{a,b,c}$, the construction of Kodiara-Spencer map with bulk insertions as well as mirror symmetry statements are work in progress of the first author together with Amorim-Hong-Lau. 
The above diagram should hold in that case also. In Section \ref{sec:333}, we will make a computation  
of $c_L$ for the case of orbi-sphere $\mathbb{P}^1_{3,3,3}$.
\end{remark}

Let us first recall relevant constructions of Fukaya-Oh-Ohta-Ono \cite{FOOO}.
First the open-closed map $OC$ is explained in Section 2, and denoted as $i_{*,qm,b^c}$ for toric manifolds in the last section.
This map is also called $p$-operator in Chapter 3 \cite{FOOOtoric3}. The Kodaira-Spencer map $\ks$ is a special case of a closed-open map $CO$ (or $i^{c*}_{qm,b^c}$).
For $\ks(w)$, consider the moduli space of $J$-holomorphic maps $u : (D,\partial D) \to (M,L)$ with interior insertion of $w$ and with an output at the boundary marked point.  In toric cases, one can take $w$ as a $T^n$-invariant cycles and the output becomes a constant multiple of fundamental class $\be_L$ (due to $T^n$-equivariancy). Pairing with $vol_L$ gives the constant and if we decorate it with the symplectic area of the disc and holonomy (or boundary deformations)), it defines the Kodaira-Spencer map $\ks(w)$. 
This (bulk) weakly unobstructedness is one of the most important ingredient for the Kodaira-Spencer map and it does not necessarily hold for general Lagrangian submanifold.
It is shown that different choice of $T^n$-invariant cycles produces the same output up to Jacobian ideal.
Fukaya-Oh-Ohta-Ono have shown that Kodaira-Spencer map is a ring isomorphism from a big quantum cohomology of toric manifold to Jacobian ring of $W$ in \cite{FOOOtoric3}(see Theorem \ref{thm:fooo1}).

The composition of these two maps, $\ks \circ OC$ is illustrated in \ref{fig:cobordism}.
On the other hand, such a configuration can be studied using a pseudo-holomorphic annuli.
The moduli space of annuli with one marked points in each boundary $\mathcal{M}_{1,1}^{ann}$ is a disc $D^2$ where
the origin of $D^2$ corresponds to sand-glass configuration and one of the boundary point of $D^2$ corresponds to the bi-crescent. Connecting these two points in $D^2$ gives a cobordism  in Figure \ref{fig:cobordism}.
They correspond to two possible degeneration of an annulus, and precise result is given in Theorem \ref{fooothm1}.

In fact, Fukaya-Oh-Ohta-Ono do not consider this composition $\ks \circ OC$ directly, but they define
$Z$-invariant, which is closely related to it. For $\ks \circ OC$, we consider pseudo-holomorphic annuli where one boundary condition is $(\bL,b)$, and the other boundary condition is $(L,\xi)$. In the toric case, we are specializing to the case that $\bL=L$, but $b$ contains mirror variables and $\xi$ is a fixed Maurer-Cartan element. Thus the result lies in a Jacobian ring.
In Fukaya-Oh-Ohta-Ono's work both boundaries of pseudo-holomorphic annuli is given by $(L,\xi)$, and the result is a numerical $Z$-invariant in $\Lambda_0$. If we fix $b$ to have value $\xi$ and take a residue value  then we recover $Z$-invariant (We will explain more about $Z$-invariant in the next section).

Using Fukaya-Oh-Ohta-Ono's theorem (Theorem \ref{fooothm1}), we can prove the following.
\begin{prop}\label{ksoc}
For general toric manifolds, we have
\begin{equation}\label{eqksoc}
\ks \circ OC (\alpha) =  \str ( \bullet \mapsto (-1)^{\frac{n(n-1)}{2} +(n-1)|\bullet|} m_2^{\xi^c,\xi^c,b^c}(\alpha,m_2^{\xi^c,b^c,b^c}(\bullet,vol_L)))
\end{equation}
where  $vol_L$ is a closed $n$-form on $L$ with  $\int_L vol = 1$.
\end{prop}
Assuming the above proposition, let us prove Theorem \ref{thm:kscommute}.
\begin{proof}
Comparing the definition of $B$-invariant in \eqref{eq:Bdefo} and \eqref{eqksoc},
the difference  lies in two expressions 
$$(-1)^{\frac{n(n+1)}{2}+1} [X_n,\cdots,X_1], (-1)^{\frac{n(n-1)}{2} } vol_L.$$

Also, from Theorem \ref{thm:vd}, \eqref{eq:Bdefo} only depend on degree $n$ part $[X_n,\cdots,X_1]^{(n)}$ of $[X_n,\cdots,X_1]$. Therefore, the difference of $B(\alpha)$ and $ \ks \circ OC(\alpha)$ is  exactly given by the scaling factor $c_L$.
$$(-1)^{\frac{n(n+1)}{2}+1} [X_n,\cdots,X_1]^{(n)} = c_L \cdot  (-1)^{\frac{n(n-1)}{2} }vol_L$$
\end{proof}
We give a proof of Proposition \ref{ksoc}.
\begin{proof}
First, note that  (since $ \langle \be_L, vol_L \rangle =1$)
$$\ks \circ OC(\alpha) =  
 \langle (\ks \circ OC(\alpha)) \be_L, vol_L \rangle_{L} =  
 \langle ( i^{c*}_{qm,b^c} \circ    i_{*,qm,\xi^c} (\alpha)), vol_L \rangle_{L}$$
 Using Theorem \ref{thm:vd} and \eqref{eq10}, this is the same as
$$= \langle ( i_{*,qm,\xi^c} (\alpha)), i_{*,qm,b^c} (vol_L)\rangle_{L} =(-1)^{\frac{n(n-1)}{2}} \sum_{I,J} g^{IJ} \langle m_2^{b^c,\xi^c,\xi^c}(e_I, \alpha), m_2^{\xi^c,b^c,b^c}(e_J, vol_L) \rangle_{PD_{L(u)}}.$$
The following lemma can be used to prove the proposition.
\begin{lemma}\label{signmukai} We have 
$$ \tr \big( \bullet \mapsto (-1)^{n |\bullet|} m_2^{\xi^c,\xi^c,b^c}(\alpha, m_2^{\xi^c,b^c,b^c}(\bullet,vol_L)) \big)
=\sum_{I,J}  g^{IJ} \langle m_2^{b^c,\xi^c,\xi^c}(e_I,\alpha),m_2^{\xi^c,b^c,b^c}(e_J,vol_L) \rangle $$
\end{lemma}
Let us explain how to prove the proposition using the lemma.
\begin{eqnarray*}
\ks \circ OC(\alpha) &=& (-1)^{\frac{n(n-1)}{2}} \sum_{I,J} g^{IJ} \langle m_2^{b^c,\xi^c,\xi^c}(e_I, \alpha), m_2^{\xi^c,b^c,b^c}(e_I, vol_L) \rangle_{PD_{L(u)}} \\
&=& (-1)^{\frac{n(n-1)}{2}} \mathrm{Tr} \big( \bullet \mapsto  (-1)^{n|\bullet|}  m_2^{\xi^c,\xi^c,b^c}(\alpha, m_2^{\xi^c,b^c,b^c}(\bullet,vol_L)) \big) \\
&=&   (-1)^{  \frac{n(n-1)}{2}}\str \big( \bullet \to   (-1)^{(n-1)|\bullet|} m_2^{\xi^c,\xi^c,b^c}(\alpha, m_2^{\xi^c,b^c,b^c}(\bullet,vol_L) \big)
\end{eqnarray*}
\end{proof}

Here is the proof of Lemma \ref{signmukai}.
\begin{proof}
We will use cyclic symmetry to prove this lemma. We first show that the boundary deformation of cyclic $\AI$-structure
still has the cyclic symmetry.

Recall that for the symmetric pairing $\langle,\rangle$ on $V$, if we set $\langle v,w\rangle_{cyc} = (-1)^{|v|(|w|+1)}\langle v,w\rangle$, then resulting pairing becomes skew-symmetric for $V[1]$.
A cyclic symmetric pairing on an $\AI$-algebra is a skew-symmetric non-degenerate pairing 
$\langle, \rangle_{cyc}$, which satisfies 
$$ \langle m_k(w_0,\cdots,w_{k-1}), w_k \rangle = (-1)^*  \langle m_k(w_1,\cdots,w_{k}), w_1 \rangle$$
where $* = |w_0|'(|w_1|'+\cdots+|w_k|')$ for the shifted degrees $|w|'=|w|-1$. 
For  boundary deformed operations, we have
$$  \langle m_1^{b,b'}(\alpha),\beta \rangle  = \sum_{k,l} m_{k+1+l}(b,\cdots,b,\alpha,b',\cdots,b'),\beta \rangle$$
$$=  (-1)^{|\alpha|'|\beta|'} \sum_{k,l} m_{k+1+l}(b',\cdots,b',\beta,b,\cdots,b),\alpha \rangle =  (-1)^{|\alpha|'|\beta|'} \langle m_1^{b',b}(\beta),\alpha \rangle$$

Similarly, one can check that
$$ \langle m_k^{b_0,\cdots,b_k}(\alpha_1,\cdots,\alpha_k),\alpha_{k+1} \rangle 
=  (-1)^{|\alpha_k|'(|\alpha_1|' + \cdots + |\alpha_{k-1}|')} \langle m_k^{b_k,b_0,\cdots,b_{k-1}}(\alpha_k,\alpha_1,\cdots,\alpha_{k-1}),\alpha_{k} \rangle $$
Therefore in the following proof, the boundary deformed version should automatically follow, and we will omit them for simplicity.

Now, we are ready to prove the lemma.
By dimension counting of the right hand side, the statement is only non-trivial if $|\alpha|=n\;  (mod\;  2)$, which we will assume. 
Let us denote the basis of $H(L,L)$ by $\{e_I\}$ with  a non-degenerate pairing $\langle e_I, e_J \rangle = g_{IJ}$ and we denote by $g^{IJ}$ its inverse matrix.
For $x = \sum_I a_Ie_I$, note that $a_I = \sum_J \langle x, g^{JI}e_J \rangle$.

Note that
$\tr \big( e_I \mapsto (-1)^{n|e_I|} m_2(\alpha, m_2(e_I,vol_L)) \big)$
equals 
\begin{eqnarray*}
\sum_{I,J} (-1)^{|e_I||vol_L|} \langle  m_2(\alpha, m_2(e_I,vol_L)), g^{JI}e_J \rangle & = &\sum_{I,J} (-1)^{\epsilon_1}  \langle  m_2(\alpha, m_2(e_I,vol_L)), g^{JI}e_J \rangle_{cyc} \\
 & =& \sum_{I,J} (-1)^{\epsilon_2} g^{JI}  \langle  m_2(e_J,\alpha), m_2(e_I,vol_L) \rangle_{cyc} \\
&=&\sum_{I,J} (-1)^{\epsilon_3} g^{JI} \langle m_2(e_J,\alpha), m_2(e_I,vol_L) \rangle \\
&=&\sum_{I,J}  g^{IJ} \langle m_2(e_I,\alpha), m_2(e_J,vol_L) \rangle 
\end{eqnarray*}
where $$\epsilon_1 =  |e_I|(|vol_L|+ |e_J|+1), \epsilon_2 = |e_I||vol_L|, \epsilon_3 = (|e_I|+n)(|\alpha|+n), \epsilon_4 = \frac{n(n-1)}{2}+n(1+|\alpha|).$$
Here, $\epsilon_1$ is computed using Lemma 3.10.9 of \cite{FOOO}. We compute exponents modulo 2.
$$ \epsilon_1 = |e_I||vol_L| + (|\alpha|+|vol_L|+ |e_I|)(|e_J|+1) = |e_I|(|vol_L|+ |e_J|+1) .$$
$\epsilon_2$ is computed using Koszul sign of cyclic symmetry.
$$\epsilon_2 = \epsilon_1 +|e_J|(|\alpha| + |m_2(e_I,vol_L)|) =\epsilon_1 + (n+|e_I|+1)(|\alpha|+1 + |e_I| + |vol_L|+1) =\epsilon_1 + |e_I|(n+|e_I|+1) $$
$\epsilon_3$ is computed using  Lemma 3.10.9 of \cite{FOOO} again.
$$\epsilon_3 = \epsilon_2 + |m_2(e_J,\alpha)|(|m_2(e_I,vol_L)| +1) = \epsilon_2 + (|\alpha|+ |e_J|)(|e_I|+|vol_L|+1) = (|e_I|+n)(|\alpha|+n)=0$$
\end{proof}

The proof of  signs in  Theorem \ref{thm:vd} was omitted in the reference, and we give its proof for completeness.
We follow the sign convention of \cite{FOOO}.
\begin{lemma}
For the case that $\dim(X)$ is even, we have
$$( M \times_L A ) \times_X B = (-1)^{\dim(B)\cdot |A|} (M \times_X B ) \times_L A $$
\end{lemma}
We omit its proof since it is a standard sign calculation as in \cite{FOOO}.
Also recall from \cite{FOOO} Section 3.10 that
\begin{equation}
\langle PD(P),PD(Q) \rangle_Z = (-1)^{|P||Q|} P \times_Z Q
\end{equation}
\begin{lemma}\label{signocco}
$$\langle OC(\alpha), h \rangle_X =\langle \alpha, CO(h) \rangle_L$$
\end{lemma}
\begin{proof}
Recall that $OC(\alpha) = \mathcal{M}_{1,1} \times_L (\alpha)$ and $\dim(OC(\alpha)) = deg(OC(\alpha))=\dim(\alpha)$.
and hence
$$\langle OC(\alpha), h \rangle_X = (-1)^{\dim(\alpha)|h|}
 OC(\alpha) \times_X h  =  (-1)^{\dim(\alpha)|h|}
  \big( \mathcal{M}_{1,1} \times_L (\alpha) \big) \times_X h$$
$$ =  
\big( \mathcal{M}_{1,1} \times_X  h \big)  \times_L \alpha=  CO(h) \times_L \alpha = (-1)^{|h||\alpha|} \langle CO(h), \alpha \rangle = \langle \alpha, CO(h) \rangle$$
The sign term in 3rd equality  vanishes  because it is  $\dim(\alpha)|h| + \dim(h)|\alpha|$. This is even since $|h| = \dim(h)$ mod 2.
\end{proof}

\section{The case with with Morse singularities}
Let $M$ be a toric manifold and $L$ a Lagrangian torus fiber. Assume that $W_L$ has Morse singularity at  $(L,b=\xi)$.
We prove the following diagram in this section, which is a reformulation of the result of Fukaya-Oh-Ohta-Ono  \cite{FOOOtoric3}. 
\begin{prop}\label{prop:m}
If $W_{L}$ has non-degenerate (Morse) singularity at $\xi$, the following diagram commutes.
\begin{equation}\label{diagramfooo}   \xymatrix{
	QH^*(M)_L  \otimes QH^*(M)_L \ar[r]^-{ \langle \cdot, \cdot \rangle_{PD_M}} \ar[d]^{\ks \otimes \ks} & \;\;\;\; \Lambda \;\;\;\; \ar@{=}[dd] \\
	  Jac(W_L) \otimes Jac(W_L) \ar[d]^{I\otimes I
	  } &   \\
	 \big(Jac(W_L) \frac{dy_1}{y_1} \cdots \frac{dy_n}{y_n} \big) \otimes \big( Jac(W_L) \frac{dy_1}{y_1} \cdots \frac{dy_n}{y_n} \big)\;\;\;\; \ar[r]^-{\langle \cdot, \cdot \rangle_{res}} &\;\;\;\; \Lambda \;\;\;\; }  
\end{equation}
The isomorphism $I$ is given by multiplying $c_L  \frac{dy_1}{y_1} \cdots \frac{dy_n}{y_n}$. If $M$ is nef, $c_L =1$, otherwise $c_L \in \Lambda_0$.
\end{prop}
\begin{remark}
Section 7 \cite{FOOOtoric1}, Section 2.2 \cite{FOOOtoric3} explains that quantum cohomology as well as
Jacobian ring localizes at each critical point $(L(u),\xi)$.  $QH^*(M)_L$ in the above diagram denotes this  component.
\end{remark}
It is still unknown whether the above diagram commutes in the non-Morse singularities.

Let us first recall the corresponding theorem from \cite{FOOOtoric3}.
\begin{theorem}\cite[Theorem 1.2.34]{FOOO}\label{thm:fooo1}
\begin{enumerate}
\item For toric manifold $M$, the Kodaira-Spencer map $ks_\frak{b}$ induces a ring isomorphism
$$ks_\frak{b}: T_\frak{b} H(M;\Lambda_0) \to Jac(W_\frak{b})$$
\item If $W_\frak{b}$ is a Morse function, then we have
$$\langle v,w \rangle_{PD_M} = \langle ks_\frak{b}(v),  ks_\frak{b}(w) \rangle_{res_Z} $$
for $v,w \in T_\frak{b} H(M;\Lambda_0)$
\item $ \langle \cdot,\cdot \rangle_{res_Z}$ equals the  standard residue pairing of complex geometry  $\langle \cdot, \cdot \rangle_{res}$
if toric manifold $M$ is nef and  bulk parameters have $deg(\frak{b}) =2$. In general, they are equal  only modulo $\Lambda_+$.
\end{enumerate}
\end{theorem}
Fukaya-Oh-Ohta-Ono introduced the new pairing on $B$-side $ \langle \cdot,\cdot \rangle_{res_Z}$ using $A$-side invariant $Z$ instead of  the  standard residue pairing of complex geometry  $\langle \cdot, \cdot \rangle_{res}$.
(More precisely, they define $ \langle \be_u,\be_u\rangle_{res_Z} = 1/Z$ so that it is compatible with quantum cohomology pairing in Lemma \ref{lem:uu}.)
In this way, the Kodaira-Spencer map becomes an isometry in (2).  In our approach, we use the standard residue pairing
on $B$-side and the difference is explicitly given by $(c_L)^2$ (where $c_L$ is the ratio of Floer volume and classical volume class defined in Definition \ref{def:fv}.)
In remark 1.2.26 of \cite{FOOOtoric3}, it is mentioned that authors  expect that pairings of $A$ and $B$-side do not agree exactly in general cases.

\subsection{$Z$-invariant and $A$-side pairing}
We recall $Z$-invariant of Fukaya-Oh-Ohta-Ono from \cite{FOOOtoric3}.
$Z$-invariant is defined for each unital Frobenius algebra $(C, \langle,\rangle, \cup, \be)$.
Given a basis $\{e_I\}$ of $C$ with $e_0$ being the unit, set $g_{IJ} = \langle e_I, e_J \rangle$ and
let $g^{IJ}$ be its inverse matrix. 
Then, the Z-invariants of $(C, \langle,\rangle, \cup, \be)$ is given by
$$ Z(C) = \sum \sum (-1)^{|e_{I_1}||e_{J_1}| + \frac{n(n-1)}{2}} g^{I_1J_1}g^{I_2J_2}g^{I_30}g^{J_30}
\langle e_{I_1}\cup e_{I_2}, e_{I_3} \rangle \langle e_{J_1}\cup e_{J_2}, e_{J_3} \rangle$$

Then, given a filtered $\AI$-algebra for a Lagrangian torus fiber $L(u)$, FOOO constructed its cyclic $\AI$-algebra 
on the canonical model $$(H(L;\Lambda_0),\{m_k^{c,\frak{b},b}\},PD[L]).$$
If one forgets higher $m_{\geq 3}$, this gives unital Frobenius algebra, hence a $Z$-invariant.

If $W$ is not Morse singularity, $Z$ invariant vanishes ($B$-invariant defined in this paper provides an alternative, since
it does not vanish for non-Morse cases).

$Z$-invariant for Morse singularities is related to the quantum cohomology pairing as follows. Let us set bulk parameter $\frak{b}=0$ for simplicity.
Theorem \ref{fooothm1}( Proposition 3.4.1 of \cite{FOOOtoric3}) implies the following identity
\begin{prop}\cite[Proposition 3.4.2]{FOOOtoric3}
Let $(u,\xi) \in Crit(W)$. Then we have
$$ \langle i_{*,qm,(\xi,u)}(vol_{L(u)}), i_{*,qm,(\xi,u)}(vol_{L(u)}) \rangle_{PD_M} = Z(\xi)$$
Here $vol_{L(u)}$ is the degree $n$ volume form on $L(u)$ with $\int_{L(u)} vol_{L(u)} =1$.
\end{prop}
\begin{lemma}\label{lem:uu}
$$\langle \be_u, \be_u \rangle_{PD_M}= \frac{1}{Z(\xi)}$$
\end{lemma}
\begin{proof}
Since $W$ is assumed to be a Morse function, its Jacobian ring is direct sum of field factors corresponding to each critical points
(determined by a torus fiber $L(u)$ and a bounding cochain $\xi$), and we denote this factor simply as $\be_{u}$
which can be identified as a class in $H(M;\Lambda_0)$.
From Theorem \ref{thm:vd}, we have
$$\langle \be_u, i_*(vol_u) \rangle_{PD_M} = \langle i^* \be_u, vol_u \rangle_{L(u)} =1.$$
Since $i^*$ is an algebra homomorphism, sending $\be_u$ to $\be_{L(u)}$, we get the last equality.
Hence $i_*(vol_u)$ may be set as $ \frac{\be_u}{\langle \be_u, \be_u \rangle_{PD_M}}$.
On the other hand, we have $Z =  \langle i_*(vol_u), i_*(vol_u) \rangle_{PD_M}$ from the previous proposition.
Hence, we get the lemma by combing these two results.
\end{proof}

\subsection{Computing $Z$-invariant}

We recall the computation of $Z$-invariant from \cite{FOOOtoric3} adapted to our setting(see Section \ref{sec:cmp}), and also follow notations thereof.
The first thing we need is to observe that Floer cohomology algebra is the Clifford algebra (\cite{Cho})

Recall that Clifford algebra $Cl(n,\vec{d})$ with $n$-generators and quadratic form $\vec{d} = \textrm{diag} (d_1,\cdots,d_n)$
is defined as $ \Lambda \ll e_1,\cdots, e_n \gg $ with relations $ e_i e_j + e_j e_i =0$ for $i \neq j$ and $e_i e_i = d_i$.
For the case of formal deformations $b = \sum x_i e_i$, we take partial derivatives of
$m(e^b) = W(b) \cdot \be$ to obtain (see \eqref{eq:dd}) $$ \frac{\partial^2 W}{\partial x_i \partial x_j} \cdot \be  = m_2^b(e_i,e_j) + m_2^b(e_j, e_i).$$
(Our mirror variable is from holonomy part only, and so this identity still holds for cyclic symmetric $\AI$-algebras.
In \cite{FOOOtoric3}, this does not hold since the bounding cochain $b^c$ may not be of degree one anymore.)
Here $\frac{\partial^2 W }{\partial x_i \partial x_j}$ is symmetric matrix and hence there exist a matrix $A$  such that
$$ A^T \frac{\partial^2 W }{\partial x_i \partial x_j} A = 2 \textrm{diag} (d_1,\cdots,d_n)$$
with $\det(A)=1$.
Therefore, by setting $e_i' = A_{ji}e_j$, $e_i$'s satisfiy the Clifford algebra relation.
And we have $e_1\wedge \cdots \wedge e_n = e_1' \wedge \cdots \wedge e_n'$.

Thus, we can define an algebra map  
$$ \Psi : Cl(n,\vec{d}) \to (H(L(u),\Lambda), \cup^Q), \Psi(e_i) = e_i'$$
To show that $\Psi$ is an isomorphism, it is enough to show that $\{e_i'\}$ generate cohomology algebra.
But since $$e_1' \cup^Q \cdots \cup^Q e_n' \equiv e_1' \wedge \cdots \wedge e_n' \;\; (\textrm{mod} \; \Lambda_+)$$
it is easy to see that $\{e_i'\}$ generate cohomology algebra $H(L(u))$ from the gapped filtered condition of Floer theory.

As explained in Section \ref{sec:cmp}, the proof of Cardy identity exploits cyclic symmetry of $\AI$-algebra. 
Fukaya-Oh-Ohta-Ono defines a quasi-isomorphic $\AI$-algebra, which has cyclic symmetry.
The isomorphism $\Phi: (H,\{m_{k,\beta}^{T^n,\rho} \}) \to (H,\{m_{k,\beta}^{c,\rho}\})$ satisfies $\Phi \equiv id  \;\; (\mod \Lambda_+)$,
and sends $m_2$-product $\cup^Q$ to $\cup^{Q,c}$.
Let us denote $\Phi(e_i') = e_i'' $. As $e_i''$ are cohomology classes in $H(L(u)) = HF(L(u))$, 
$e_1'' \cup^{Q,c} \cdots \cup^{Q,c} e_n''$ should be a cohomology class in $H(L(u))$.
We denote its top degree component as $ \big(e_1'' \cup^{Q,c} \cdots \cup^{Q,c} e_n''\big)^{(n)}$

\begin{definition}
We define $c_L \in \Lambda_0$ as 
$$
 \big(e_1'' \cup^{Q,c} \cdots \cup^{Q,c} e_n''\big)^{(n)} = c_L \cdot  e_1 \wedge \cdots \wedge e_n$$
Here $c_L = 1+ \sigma$ for some $\sigma \in \Lambda_+$ by construction.
\end{definition}
Now, let us explain how Fukaya-Oh-Ohta-Ono compute $Z$-invariant from $\{e_1'',\cdots, e_n''\}$.
We may set $$e_I'' = e_{i_1}'' \cup^{Q,c} \cdots \cup^{Q,c} e_{i_k}''$$
for $I = \{i_1,\cdots, i_k\}$ with $i_1 \leq \cdots \leq i_k$.
Note that we have 
$$e_I'' \cup e_{\{1,\cdots,n\}}'' = \pm (\prod_{ i \in I} d_i) e_{I^c}''$$

Here is the computation of $Z$-invariant(we refer readers to \cite{FOOOtoric3} for the signs).
Set \[\langle e_I'', e_J'' \rangle = g_{IJ}\] and $g^{IJ}$ to be its inverse.

\begin{eqnarray*}
\sum \pm g^{IJ} \langle  m_2(e_I'', e_{\{1,\cdots,n\}}''), m_2(e_J'', e_{\{1,\cdots,n\}}'') \rangle 
 &= & \sum \pm g^{IJ} \langle  m_2\big( e_{\{1,\cdots,n\}}'', m_2(e_J'', e_{\{1,\cdots,n\}}'')\big), e_I'' \rangle \\
 &=& \sum \pm g^{IJ} \langle (\prod_{ i \in J} d_i) m_2\big( e_{\{1,\cdots,n\}}'',  e_{J^c}'' \big), e_I'' \rangle  \\
  &=&  \sum \pm g^{IJ} \langle (\prod_{ i \in J} d_i) (\prod_{ i \in J^c} d_i)\langle e_J'', e_I'' \rangle \\
&= & 2^n d_1\cdots d_n.
\end{eqnarray*}

From Theorem \ref{fooothm1}, we have (since  $\det(Hess (W) )= 2^n d_1\cdots d_n$)
\begin{equation}\label{eq:ii}
 \langle i_*( e_{\{1,\cdots,n\}}''),  i_*( e_{\{1,\cdots,n\}}'') \rangle_{PD_X} =  \det(Hess (W) ).
 \end{equation}

Now, we prove the proposition \ref{prop:m},
$$\langle v,w \rangle_{PD_M} =\big( c_L\big)^2 \langle \ks(v),  \ks(w) \rangle_{res}. $$
\begin{proof}
By theorem \ref{thm:vd}, the LHS of \eqref{eq:ii} only depends on degree $n$ part of $e_{\{1,\cdots,n\}}''$. So we have
$$(c_L)^2  \cdot  \langle i_*(vol_u), i_*(vol_u) \rangle_{PD_M} = (c_L)^2 \cdot Z $$
So equation  \eqref{eq:ii}  is the same as  $(c_L)^2 \cdot Z =  \det(Hess (W) )$. 
Therefore, we have
$$\langle 1_u, 1_u \rangle_{PD_M}  = \frac{1}{Z(L(u),b)} =  \frac{(c_L)^2}{ \det(Hess (W) )} = (c_L)^2 \langle 1_u,1_u \rangle_{res}. $$
\end{proof}

\section{Mirror symmetry of residue pairings: $\PP^1_{3,3,3}$}\label{sec:333}
Let $\tilde{X}=\C/(\Z\oplus e^{2\pi i/3}\Z)$ be a symplectic torus, and consider a $\Z/3$-action on $\tilde{X}$ as the rotation by $2\pi/3$. Its $\Z/3$-quotient is an orbifold sphere $\PP^1_{3,3,3}$ with three $\Z/3$-orbifold point.
There is an immersed Lagrangian $\LL$, called {\em Seidel Lagrangian} in  $\PP^1_{3,3,3}$
whose  Floer potential is the cubic polynomial $ W_\LL=\phi(x^3-y^3+z^3)+\psi xyz$ for some $\phi,\psi \in \Lambda$.
Homological mirror symmetry between $\PP^1_{3,3,3}$ and $W_\LL$ using localized mirror functor has been 
proved in \cite{CHL}(see below for more details).
Closed mirror symmetry between quantum cohomology of  $\PP^1_{3,3,3}$ and $Jac (W_\LL)$ is a work in progress of
Amorim-Cho-Hong-Lau.

In this section we prove the following theorem.
\begin{theorem}\label{thm:9}
The Poincar\'e pairing for quantum cohomology and the standard residue pairing of Jacobian ring are conformally equivalent.
\begin{equation}\label{rescomm}
\xymatrix{
QH^*(\PP^1_{3,3,3})\otimes QH^*(\PP^1_{3,3,3}) \ar[rr]^-{\langle\;,\;\rangle_{PD}} \ar[d]_{\ks \otimes \ks}& &\Lambda \ar@{=}[dd] \\
Jac(W_\LL) \otimes Jac(W_\LL) \ar[d]_{I \otimes I} &&\\
Jac(W_\LL) dx dy dz \otimes Jac(W_\LL) dx dy dz    \ar[rr]^-{\langle\;,\;\rangle_{res}} & &\Lambda}
\end{equation}
Here, $I$ is given by multiplication of $c_\LL dx dy dz$, where 
\[ c_\LL=\sum_{k\geq 0} (-1)^k \big(q^{(6k+1)^2}-q^{(6k+5)^2}\big)\]
with $q := T^{\omega(\Delta)}$, where $\omega(\Delta)$ is the area of minimal triangle.
\end{theorem}
%
We first recall that $\LL$ is constructed as a quotient of 
an embedded circle  $\widetilde{\LL}_1$ in  $\tilde{X}$.
Consider the image of $\iota: \R \to \C$ be a map \[ t \mapsto \frac{1+e^{2\pi \sqrt{-1}/3}}{2}+\sqrt{-1}t\] and $pr: \C \to \tilde{X}$ be the quotient map. Then $pr\circ \iota$ defines an embedded circle $\widetilde{\LL}_1$.
The immersed Lagrangian $\LL$ is obtained from $\widetilde{\LL}_1$ projecting to the quotient $\PP^1_{3,3,3}$.
The preimage of $\LL$ in $\tilde{X}$ is given by three embedded circles, which are  $\Z/3$-action images of $\widetilde{\LL}_1$. Furthermore, we equip $\LL$ with $\C$-bundle the holonomy $\lambda=e^{2 \pi \sqrt{-1}}$ (or non-trivial spin structure)
represented by a point on $\LL$ so that we reverse the sign whenever the boundary passes through this point.

%

Then we have three degree 1 immersed sectors $X$, $Y$ and $Z$ for Floer homology of $\LL$ in $\PP^1_{3,3,3}$. It is proved by \cite{CHL} that their linear combinations are weak bounding cochains:
\begin{theorem}\cite{CHL}
Let $b:=xX+yY+zZ$, where $x$, $y$ and $z$ are formal variables. Then the Landau-Ginzburg potential is given by
\[ W_\LL(b)=\phi(x^3-y^3+z^3)+\psi xyz \in \Lambda[x,y,z]\]
where $\phi$ and $\psi$ are formal power series of $q$ as follows:
\[ \phi(q)=\sum_{k\in \mathbb{Z}} (-1)^{k+1}(k+\frac{1}{2})q^{(6k+3)^2},\]
\[ \psi(q)=\sum_{k \in \mathbb{Z}} (-1)^{k+1} (6k+1)q^{(6k+1)^2}.\]
\end{theorem}

Using the family $(\LL,b)$, we have the closed string mirror symmetry as follows.
\begin{theorem}[Amorim-Cho-Hong-Lau]
There exists a Kodaira-Spencer map  $$\ks: QH^*(\PP^1_{3,3,3}) \to Jac(W_\LL)$$ which is a ring isomorphism.
\end{theorem}
\begin{remark}
In fact, in the above work in progress, they discuss all cases $\PP^1_{a,b,c}$ with bulk deformation by twisted sectors.
\end{remark}

Let $1_{\PP^1_{3,3,3}},\; PD(pt) \in QH^*(\PP^1_{3,3,3})$ be Poincar\'{e} duals of the fundamental class and the point class of $\PP^1_{3,3,3}$, respectively. Then 
\[ \langle 1_{\PP^1_{3,3,3}},PD(pt)\rangle_{PD}=\int_{\PP^1_{3,3,3}}1_{\PP^1_{3,3,3}} \cup_Q PD(pt)=1.\] In \cite{CHL}, it is shown that 
\[\ks(1_{\PP^1_{3,3,3}})=1,\; \ks(PD(pt))=\frac{1}{8}\big(q \frac{\partial W_\LL}{\partial q}\big)=\frac{q}{8}\cdot \frac{\partial \phi}{\partial q}(x^3-y^3+z^3)+\frac{q}{8}\cdot\frac{\partial \psi}{\partial q}xyz.\]

We may compute and check  $\langle \ks(1_{\PP^1_{3,3,3}}),\ks(PD(pt))\rangle_{res} \neq 1$  that Kodaira-Spencer map $\ks$ does not preserve the pairings. 

Let $L$ be an immersed Lagrangian which is same as $\LL$ as underlying manifolds, with bounding cochain $\xi=0$.
We can define $B$-invariant as in section \ref{sec:B}.
As explained in previous sections, we may consider the diagram
\begin{equation}\label{cardy}\xymatrix{
HF(L,L) \ar[r]^{OC} \ar[rdd]_{B} & QH^*(\PP^1_{3,3,3}) \ar[d]^{\ks} \\
& Jac(W_\LL)\ar[d]^I \\
& Jac(W_\LL) dx dy dz}\end{equation}
We define the module isomorphism $I$ as in Theorem \ref{thm:kscommute} using the ratio $c_\LL$ of Floer volume class and classical volume class .  Namely, we set 
\[ [Z,Y,X]^{(top)}=c_\LL vol_\LL\]
for some $c_\LL \in \Lambda_0$ (recall that $[Z,Y,X]=m_2^b(m_2^b(Z,Y),X)$). The notation $(top)$ means that we project $[Z,Y,X]$ to the top degree component.

\begin{lemma}\label{lem:cl}
\[ [Z,Y,X]=c_\LL vol_\LL+ c_1 xX+ c_2 yY+ c_3 zZ\] for some $c_\LL,c_1,c_2,c_3 \in \Lambda_0$, where
\[c_\LL=\sum_{k\geq 0} (-1)^k \big(q^{(6k+1)^2}-q^{(6k+5)^2}\big)=\sum_{k \in \mathbb{Z}} (-1)^k q^{(6k+1)^2}.\]
\end{lemma}
\begin{proof}
$m_2^b(Z,Y)$ is given by counting of holomorphic triangles whoses vertices cyclically ordered by $(Z,Y,X)$, $(Z,Y,xX)$, $(Z,xX,Y)$ or $(xX,Z,Y)$.
Here, $(Z,Y,X)$ denotes  triangles without $b$-insertion which has two inputes $Z, Y$ and one output $\bar{X}$.
The other expressions include a term $xX$, which comes from the weak bounding cochain $b$,
corresponds to an operation with an output given by a multiple of a unit $1_\LL$.
%

Hence, $m_2^b(Z,Y)=c_\LL \bar{X}+c'x \be_\LL$, and
\[ [Z,Y,X]= m_2^b(c_\LL \bar{X}+c'x1_\LL,X)=m_2^b(c_\LL\bar{X},X)+c'xX.\]
We divide \[m_2^b(c_\LL \bar{X},X)=m_{2,0}(c_\LL \bar{X},X)+m_{2,+}^{b}(c_\LL \bar{X},X)\]
where $m_{2,0}$ computes the component coming from a constant triangle, while $m_{2,+}^{b}$ are components of positive areas. Note that 
\[m_{2,0}(\bar{X},X)=vol_\LL,\] so $m_{2,0}(c_\LL \bar{X},X)=c_\LL vol_\LL$. 

On the other hand, $m_{2,+}^{b}$ involves nontrivial weak bounding cochain contributions. By degree reason, we have only one weak bounding cochain for each polygon contributing to $m_{2,+}^{b}$, so the problem boils down to the count of holomorphic trapezoids(i.e. computation of $m_3$) whose vertices are ordered as follows, where the last vertices are (duals of) outputs:
\begin{itemize}
\item $(zZ,\bar{X},X,\bar{Z})$
\item $(yY,\bar{X},X,\bar{Y})$
\item $(xX,\bar{X},X,\bar{X})$
\item $(\bar{X},zZ,X,\bar{Z})$
\item $(\bar{X},yY,X,\bar{Y})$
\item $(\bar{X},xX,X,\bar{X})$
\item $(\bar{X},X,zZ,\bar{Z})$
\item $(\bar{X},X,yY,\bar{Y})$
\item $(\bar{X},X,xX,\bar{X})$

\end{itemize}
It is evident that there are no other trapezoids which contribute to $m_2^{b_+}(\bar{X},X)$, from Figure \ref{fig:m2xbarx}.
\begin{figure}
	\includegraphics[height=3in]{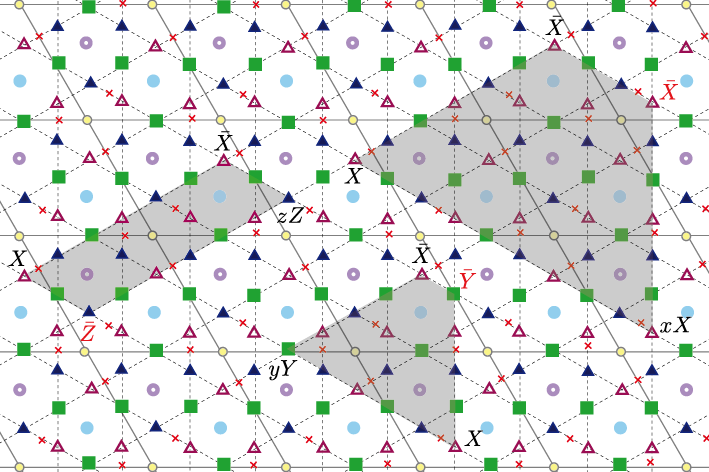}
	\caption{We depict three trapezoids, with vertices $(zZ,\bar{X},X,\bar{Z})$, $(\bar{X},yY,X,\bar{Y})$ and $(\bar{X},X,xX,\bar{X})$. Vertices which represent outputs($Z$, $Y$ and $X$ respectively) are in red color. }
	\label{fig:m2xbarx}
\end{figure}

Hence, we have
\[ m_2^b(m_2^b(Z,Y),X)=c_\LL vol_\LL+ c_1 xX+c_2 yY+ c_3 zZ\] for some $c_\LL,c_1,c_2,c_3 \in \Lambda_0$. 

Now we compute $c_\LL$ which is the number of triangles with vertices $(Z,Y,X)$. Recall that signs for holomorphic polygons on surfaces are computed as follows, due to \cite{Se2}. Let $u \in \mathcal{M}(p_1,...,p_k;q)$ be a holomorphic polygon, where $p_1 \in L_0 \cap L_1,\cdots, p_k \in L_{k-1}\cap L_k, q \in L_k \cap L_0$. The sign of $u$ is determined by the following steps.
\begin{itemize}
 \item If a Lagrangian is equipped with a nontrivial spin structure, put a point $\circ$ on it, on which the nontrivial spin bundle is twisted.
 \item Disagreement of the orientation of $\partial u$ on $L_0$ is irrelevant.
 \item If the orientation of $\partial u$ on $\widearc{p_i p_{i+1}}$ does not agree with $L_i$, the sign is affected by $(-1)^{|p_i|}.$
 \item If the orientation of $\partial u$ on $\widearc{p_k q}$ does not agree with $L_k$, the sign is affected by $(-1)^{|p_k|+|q|}.$ 
 \item Mutiply $(-1)^l$ when $\partial u$ passes through nontrivial spin points $\circ$ $l$ times.
\end{itemize}
\begin{figure}
	\includegraphics[height=3in]{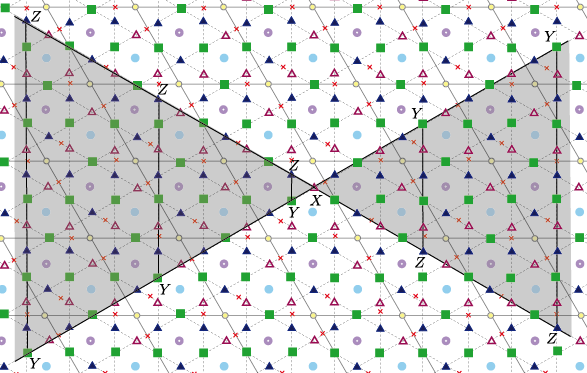}
	\caption{The computation of $c_\LL$. Spin structures are denoted by red crosses on $\LL$.}
	\label{fig:cL}
\end{figure}

The Figure \ref{fig:cL} illustrates all triangles of vertices $(Z,Y,X)$. The triangles on the left have areas as $q^1, q^{(6+1)^2},\cdots, q^{(6k+1)^2},\cdots$. Their signs are positive before considering spin structures, but the number of spin structures changes alternatingly. Hence, the triangles on the left gives the following number
\[ \sum_{k\geq 0} (-1)^k q^{(6k+1)^2}.\]
 We also count triangles on the right, and obtain the number \[ \sum_{k\geq 0} (-1)^{k+1} q^{(6k+5)^2}\]
 and sum of those numbers gives the formula of $c_\LL$.
\end{proof}

  By definition, for $\alpha \in HF(L,L)$, we have
\[ B(\alpha)= \str \big(m_2^{0,0,b}(\alpha,m_2^{0,b,b}(\bullet,[Z,Y,X])) \big)dx\wedge dy\wedge dz. \]

Now, we prove an analogue of Corollary \ref{co:degb}. We remark that  we can also use
another approach (as in toric case) that the output of $\ks$ map is a multiple of unit to prove it.
\begin{lemma}\label{lem:degb}
\[B(\alpha)=\str\Big( m_2^{0,0,b}\big(\alpha, m_2^{0,b,b}(\bullet,c_\LL vol_\LL)\big)\Big)dx\wedge dy\wedge dz.\]
In particular, the terms  $c_1 xX+ c_2 yY+ c_3 zZ$ in $[Z,Y,X]$ do not contribute to the $B$-invariant.

\end{lemma}
\begin{proof}
 By Theorem \ref{perminv}, we have
\[ \str \Big( m_2^{0,0,b}\big(\alpha, m_2^{0,b,b}(\bullet,m_2^{b,b,b}(m_2^{b,b,b}(Z,Y),X))\big)\Big)
=-\str \Big( m_2^{0,0,b}\big(\alpha, m_2^{0,b,b}(\bullet,m_2^{b,b,b}(m_2^{b,b,b}(Y,Z),X))\big)\Big).\]
On the other hand, we have
\[ m_2^{b,b,b}(Z,Y)+m_2^{b,b,b}(Y,Z)=\frac{\partial^2 W_\LL}{\partial y \partial z}\cdot {\be}=\psi x \cdot {\be},\]
so 
\begin{align*}
 &\str \Big( m_2^{0,0,b}\big(\alpha, m_2^{0,b,b}(\bullet,m_2^{b,b,b}(m_2^{b,b,b}(Z,Y)+m_2^{b,b,b}(Y,Z),X))\big)\Big)\\
 =&\;\str \Big( m_2^{0,0,b}\big(\alpha, m_2^{0,b,b}(\bullet,m_2^{b,b,b}(\psi x\cdot {\be},X))\big)\Big)\\
 =&\;\str \Big( m_2^{0,0,b}\big(\alpha, m_2^{0,b,b}(\bullet,\psi xX)\big)\Big)=0
 \end{align*} 
 in $Jac(W)$. Similarly we can show that
 \[ \str\Big(m_2^{0,0,b}\big(\alpha,m_2^{0,b,b}(\bullet,\psi yY)\big)\Big)=\str\Big(m_2^{0,0,b}\big(\alpha,m_2^{0,b,b}(\bullet,\psi zZ)\big)\Big)=0.\]
 Hence, $B(\alpha)=\str\Big( m_2^{0,0,b}\big(\alpha, m_2^{0,b,b}(\bullet,c_\LL vol_\LL)\big)\Big)dx\wedge dy\wedge dz$.
 \end{proof}
Assuming Cardy identity (which is a work in progress), we have  the commutativity of (\ref{cardy}), 
\[B(\alpha)=c_\LL\cdot (\ks \circ OC)(\alpha)dx\wedge dy\wedge dz,\]
We define a  new pairing on $Jac(W_\LL)$ as follows:
\[ \langle f,g \rangle_{c_\LL}:=\langle c_\LL\cdot f, c_\LL\cdot g\rangle_{res} = c_\LL^2 \langle f, g \rangle_{res}.\] 

The following proposition proves the Theorem \ref{thm:9}.
\begin{prop}\label{prop:isoks}
$\ks$ is an isometry between $(QH^*(\PP^1_{3,3,3}),\langle\;,\;\rangle_{PD})$ and $(Jac(W_\LL),\langle\;,\;\rangle_{c_\LL})$.
\end{prop}

\begin{proof}
Recall that for $f,g \in Jac(W_\LL)$,
\begin{align}\langle f,g\rangle_{res}&=Res|_{\partial_x W_\LL=\partial_y W_\LL=\partial_z W_\LL=0} \left(\begin{array}{c}fg \\ \partial_x W_\LL \cdot\partial_y W_\LL \cdot\partial_z W_\LL \end{array}\right)\nonumber \\
&=Res|_{x=y=z=0} \left(\begin{array}{c}fg\cdot \det C \\ x^l y^l z^l\end{array}\right) \label{residue}
\end{align}
which is nothing but the coefficient of $x^{l-1}y^{l-1}z^{l-1}$ of $fg\cdot \det C$, when $C$ is a matrix satisfying
\[ C\cdot \left(\begin{array}{c}\partial_x W_\LL \\ \partial_y W_\LL \\ \partial_z W_\LL\end{array}\right)=\left(\begin{array}{c}x^l \\ y^l \\ z^l\end{array}\right)\] for some $l$. Define $C$ as
\[ \frac{1}{3\phi (27\phi^3-\psi^3)}\left(\begin{array}{ccc} (27\phi^3-\psi^3)x^2-9\phi^2 \psi yz & 3\phi \psi^2 z^2 & \psi^3 xz \\
-\psi^3 xy & (27\phi^3-\psi^3)y^2+9\phi^2 \psi xz & 3\phi \psi^2 x^2 \\ 3\phi \psi^2 y^2 & -\psi^3 yz & (27\phi^3-\psi^3) z^2-9\phi^2 \psi xy \end{array}\right),\] then
\[ C\cdot \left(\begin{array}{c}\partial_x W_\LL \\ \partial_y W_\LL \\ \partial_z W_\LL\end{array}\right)=\left(\begin{array}{c}x^4 \\ y^4 \\ z^4 \end{array}\right)\]
and 
\[ \det C=\frac{27\phi^3(27\phi^3-\psi^3)^2\cdot x^2 y^2 z^2+9\phi^2 \psi(27\phi^3-\psi^3)^2(-x^3y^3+x^3z^3-y^3z^3)}{(3\phi(27\phi^3-\psi^3))^3}+r(x,y,z)\]
where $r(x,y,z) \in O(x^4)+O(y^4)+O(z^4)$. Hence $r(x,y,z)$ contributes as $0$ in (\ref{residue}).

Our goal is the following: $\langle a,b\rangle_{PD}=\langle \ks(a),\ks(b)\rangle_{c_\LL}$, in other words
\[\int_{\PP^1_{3,3,3}}a\cup_Q b=c_\LL^2 \cdot Res|_{\partial_x W_\LL=\partial_y W_\LL=\partial_z W_\LL=0}\left(\begin{array}{c}\ks(a)\ks(b) \\ \partial_x W_\LL \cdot\partial_y W_\LL \cdot\partial_z W_\LL \end{array}\right),\]
and if $a\cup_Q b=PD(pt)+\alpha$ where $\alpha \in QH^*(\PP^1_{3,3,3})$ is of lower degree than $PD(pt)$, then the right hand side is 1. Since $\ks$ is a ring homomorphism, $\ks(a)\ks(b)=\ks(a\cup_Q b)$. By Cho-Hong-Lau's result, $\ks(\alpha)$ has degree less than 3, so it contributes as zero in the residue. To summarize, we finish the proof if we show that
\begin{equation}\label{powerseriesid} c_\LL^2\cdot Res|_{\partial_x W_\LL=\partial_y W_\LL=\partial_z W_\LL=0}\left(\begin{array}{c}\ks(PD(pt)) \\ \partial_x W_\LL \cdot\partial_y W_\LL \cdot\partial_z W_\LL \end{array}\right)=1.\end{equation}

Recall from above that 
\[\ks(PD(pt))=\frac{1}{8}q \frac{\partial \phi}{\partial q}(x^3-y^3+z^3)+\frac{1}{8}q\frac{\partial \psi}{\partial q}xyz,\]
so
\begin{align*}
& c_\LL^2 \cdot Res|_{\partial_x W_\LL=\partial_y W_\LL=\partial_z W_\LL=0}\left(\begin{array}{c}\ks(PD(pt)) \\ \partial_x W_\LL \cdot\partial_y W_\LL \cdot\partial_z W_\LL \end{array}\right)\\
=& \;c_\LL^2\cdot Res|_{x=y=z=0} \left(\begin{array}{c}\det C\cdot \big(\frac{1}{8}q \frac{\partial \phi}{\partial q}(x^3-y^3+z^3)+\frac{1}{8}q\frac{\partial \psi}{\partial q}xyz\big) \\ x^4 y^4 z^4 \end{array}\right) \\
=& \;c_\LL^2\cdot \frac{27\phi^3(27\phi^3-\psi^3)^2\cdot q\frac{\partial \psi}{\partial q}-27 \phi^2 \psi(27\phi^3-\psi^3)^2 \cdot q \frac{\partial \phi}{\partial q}}{8\cdot (3\phi(27\phi^3-\psi^3))^3}. 
\end{align*}
Therefore, \eqref{powerseriesid} corresponds to the following identification of formal power series.
\begin{equation}\label{eq:nn}
c_\LL^2\Big(\phi \cdot q\frac{\partial \psi}{\partial q}-\psi\cdot q\frac{\partial \phi}{\partial q}\Big)=8\phi(27\phi^3-\psi^3),
 \end{equation}
This identity is proved in Appendix \ref{app2}.
\end{proof}

\begin{remark}
Satake-Takahashi \cite{ST} proved that Frobenius manifold structures on $QH^*(\PP^1_{3,3,3})$ and on the universal unfolding of $W_\LL$ are isomorphic, using uniqueness of solution of WDVV equation. But it was not known whether the
Kodaira-Spencer map(which is constructed geometrically) preserves the parings on each side. We showed that it does not preserve the pairing at the origin unless we multiply $c_\LL^2$. 
If we generalize our construction using bulk-deformations, we obtain universal unfoldings of $W_{\LL}$. 
We hope to investigate such a generalization and relation to the work of \cite{ST} in the near future.
\end{remark}

\appendix

\section{Properties of supertrace}
For $\Z/2$-graded vector spaces $E^\bullet = E^0 \oplus E^1$, and $F^\bullet=F^0 \oplus F^1$, a morphism $\Phi: E^\bullet \to F^\bullet$ can be
written as 
 \[ \Phi= \left[ {\begin{array}{cc} \Phi_{00} & \Phi_{01} \\ \Phi_{10} & \Phi_{11} \ \end{array} } \right], \; \textrm{where} \; \Phi_{ij} \in \Hom(E^j,F^i). \]
\begin{definition}
A {\em supertrace }of a morphism $\Phi$ is defined as
$$\str(\Phi) = \mathrm{Tr}(\Phi_{00}) - \mathrm{Tr}(\Phi_{11}).$$
\end{definition}
The following is easy to check.
\begin{lemma}\label{lem:stp}
\begin{enumerate}
\item  $\str(\Phi) = 0$ if the degree of $\Phi$ is odd. 
\item $\str(\Phi + \Psi) = \str(\Phi) + \str(\Psi)$
\item
$\str(\Phi \Psi) = (-1)^{|\Phi||\Psi|} \str(\Psi \Phi)$
for homogeneous morphisms $\Phi$ and $\Psi$.
\item In particular, we have $\str([\Phi,\Psi]) = 0$ where $[\Phi,\Psi] = \Phi\Psi - (-1)^{|\Phi||\Psi|} \Psi  \Phi$.
\item $\str([A,B]C) = (-1)^{1+ |A||B|}\str(B[A,C]) = (-1)^{|C|(|A|+|B|)} \str([C,A]B)$
\end{enumerate}
\end{lemma}
Here are some elementary properties of boundary-bulk maps.
\begin{lemma}\label{strprop}
 $\str(f (dQ)^n )$ satisfies the following properties.
 \begin{enumerate}
 \item $\str(f (dQ)^n) =0$ unless $|f| +n \equiv 0$ mod 2.
 \item $\str(\delta(f)(dQ)^n) =0$ modulo $\partial_i W$'s.
 \item $\str(A(\delta(B))) = -(-1)^{|A|}\str(\delta(A)B)$
 \end{enumerate}
\end{lemma}
\begin{proof}
The first and third  follow from the basic properties in Lemma \ref{lem:stp}. For the second property, 
differentiate $Q^2 = W \cdot Id$ to obtain
$(dQ)Q + Q(dQ) =dW \cdot Id$. Hence, $dQ$ and $Q$ anti-commute modulo $\partial_i W$'s.
$$\str((Q f  - (-1)^{|f|} f Q ) (dQ)^n )  =
\str (Qf(dQ)^n - (-1)^{|f|+n} f (dQ)^nQ) = 0.$$
\end{proof}

\section{Mirror identity}\label{app2}
This appendix \ref{app2} is due to Dohyeong Kim.
For \begin{align}
\phi(q) &= \sum_{k \in \Z} (-1)^k (k + \frac 1 2 ) q^{(6k+3)^2}
\\
\psi(q) &=  \sum_{k \in \Z} (-1)^k (6k + 1 ) q^{(6k+1)^2}
\\
c_{\mathbb L}(q) &= \sum_{k \in \Z} (-1)^k q^{(6k+1)^2},
\end{align}
we show the following identity, which is used in the proof of Proposition \ref{prop:isoks}.
\begin{prop} 
We have the following identity
\begin{align}\label{mirrorid}
c_{\mathbb L}^2 \left(\phi q \partial_q \psi - \psi q \partial_q \phi \right) = 8 \phi \left( 27 \phi^3 - \psi^3 \right).
\end{align}
\end{prop}
\begin{remark}
The derivative  $\left(\phi q \partial_q \psi - \psi q \partial_q \phi \right)$ is an example of Rankin-Cohen Bracket  of modular forms. For more details, see Rankin \cite{R57}.
\end{remark}
\begin{proof}
Basic idea of proof is to show that each side of the identity  is a modular form in $M_6(\Gamma_0(576))$,
and argue that the identity can be shown by checking the equation up to $q^{577}$. We explain the details in the rest of the appendix.

\subsection{Preliminaries on theta series and Dirichlet character}
Shimura \cite{Shim} worked on modular forms with general half-integral weight  and  also with level, character and multiplier systems. Recall that theta series with a Dirichlet character $\chi$ is defined as 
 \begin{align}
\theta_\chi := \sum_{n \in \Z } \chi(n) q^{n^2}
\\
\Theta_\chi := \sum_{n \in \Z } \chi(n) n q^{n^2}.
\end{align}
We will need a transformation formula of theta series given below.

Let $\gamma \in \Gamma_0(4)$. Let
$$\gamma = \begin{bsmallmatrix}a&b\\c&d\end{bsmallmatrix}$$
be the matrix coefficients of $\gamma$. The automorphy factor $j(\gamma,\tau)$ is defined by
$$
j(\gamma,\tau) = \theta_1(\gamma \tau) / \theta_1(\tau).
$$
 It is so-called the theta multiplier system. Then by \cite{Shim} \ p.\,447 , we have
$$
j(\gamma,\tau) = \epsilon_d^{-1} \times \left( \frac{c}{d}\right) \times (c\tau+d)^{1/2}
$$
where the square-root is taken to have positive real part, $\epsilon_d$ is a fourth root of unity, and $\left( \frac{c}{d}\right)$ is the Legendre symbol. In particular,
\begin{align}\label{j}
j(\gamma,\tau)^4 =  (c\tau+d)^2.
\end{align}

\begin{prop}[Proposition 2.2\cite{Shim}]\label{shimura}
Let $\chi$ be a primitive Dirichlet character of conductor $r$. Let $\gamma = \begin{bsmallmatrix}a&b\\c&d\end{bsmallmatrix} \in \Gamma_0(4r^2)$. Then, we have
\begin{align}
\Theta_\chi(\gamma \tau) &= \chi(d)  \left( \frac{-1}{d}\right) j(\gamma,\tau)^3 \Theta_\chi(\tau) \\
\theta_\chi(\gamma \tau) &= \chi(d) j(\gamma,\tau) \theta_\chi(\tau).
\end{align}

\end{prop}
 We will use the following Dirichlet characters in this appenix. For any $n \in \Z$, 
$$
\chi_4(n) = 
\begin{cases}
1 & \text{if $n=4k+1$}
\\
-1&\text{if $n=4k+3$}
\\
0&\text{otherwise,}
\end{cases}
$$
$$
\tilde \chi_{4}(n) = 
\begin{cases}
1 & \text{if $n=12k+1$}
\\
1&\text{if $n=12k+5$}
\\
-1&\text{if $n=12k+7$}
\\
-1&\text{if $n=12k+11$}
\\
0&\text{otherwise,}
\end{cases}
$$
and
$$
\chi_{12}(n) = 
\begin{cases}
1 & \text{if $n=12k+1$}
\\
-1&\text{if $n=12k+5$}
\\
-1&\text{if $n=12k+7$}
\\
1&\text{if $n=12k+11$}
\\
0&\text{otherwise.}
\end{cases}
$$
Then we have
$$
\tilde \chi_{4}(n) + \chi_{12}(n)= 
\begin{cases}
2 & \text{if $n=12k+1$}
\\
-2&\text{if $n=12k+7$}
\\
0&\text{otherwise.}
\end{cases}
$$
 Recall that $\chi$ is even if $\chi(-1)=1$, odd if $\chi(-1)=-1$.  Here is an elementary observation.
\begin{lemma}
Suppose $\chi$ is odd. Then $\theta_\chi=0$. Similarly, if  $\chi$ is even, then $\Theta_\chi=0$. 
\end{lemma}
\begin{proof}
Consider $n \mapsto -n$. 
\end{proof}
Theta series associated to $\chi_4$ and $\tilde \chi_4$ are essentially the same. Note that
$$
(3,n)=1 \Rightarrow \chi_4(n) = \tilde \chi_4(n).
$$
On the other hand, if $(3,n) \not = 1$, then $n= 3m$ and $\tilde \chi_4 (n) = 0$. One can rewrite it as
\begin{align}
\theta_{\chi_4}(q)-  \theta_{\tilde \chi_{4}}(q) &= \sum_{m \in \Z} \chi_4(3m) q^{(3m)^2}
\\
&= \chi_4(3)\sum_{n \in \Z} \chi_4(n) q^{9 n^2}
\\
&= \chi_4(3) \theta_{\chi_4}(q^9)
\\
&= -\theta_{\chi_4}(q^9),
\end{align}
so
$$
\theta_{\tilde \chi_{4}} (q) = \theta_{\chi_4}(q) + \theta_{\chi_4}(q^9).
$$
Similarly, we have
\begin{align}\label{eq-old}
\Theta_{\tilde \chi_{4}} (q) = \Theta_{\chi_4}(q) + \Theta_{\chi_4}(q^9).
\end{align}

\subsection{Writing $\phi, \psi, c_{\mathbb L}$ as theta series}
The key point is that the series $\phi$, $\psi$, and $c_{\mathbb L}$ are essentially linear combinations of these theta series. 
\begin{prop}\label{theta1} We have
\begin{eqnarray*}
 \phi(q) &= &\frac{1}{2} \Theta_{\chi_4}(q^9) \\
 \psi(q) &=& \frac{1}{2} \Theta_{\tilde \chi_4}(q) \\ 
 c_{\mathbb L}(q) &=& \frac{1}{2} \theta_{\chi_{12}}(q)
 \end{eqnarray*}
\end{prop}
\begin{remark}
Lau-Zhou \cite{LZ} observed that $\phi, \psi$ can be written in terms of eta functions
$$\phi(q) = - \eta(q^3)^3, 
\psi(q) =- \left( \eta(q^{1/3})^3 + 3\eta(q^3)^3\right) $$
\end{remark}
\begin{proof}
For the case of $\phi$, we have
\begin{align}
\Theta_{\chi_4} &= \sum_{n \in \Z } \chi_4(n) nq^{n^2}
\\
&= \sum_{k \in \Z } (-1)^k (2k+1)q^{(2k+1)^2}
\\
&=  2\cdot \phi(q^{1/9}).
\end{align}

Equivalently,
$$
\phi(q) = \frac 1 2 \Theta_{\chi_4}(q^9).
$$

For the case of $\psi$, we have
\begin{align}
\Theta_{\tilde \chi_4}+\Theta_{\chi_{12}} &= \sum_{n \in \Z } \left( \tilde \chi_4(n) + \chi_{12}(n) \right) nq^{n^2}
\\
&=2 \times \sum_{k \in \Z }(-1)^k(6k+1)q^{(6k+1)^2}
\\
&=2\psi(q).
\end{align}
However, the parity of $\chi_{12}$ being even forces $\Theta_{\chi_{12}}=0$. Thus, we have
$$
\Theta_{\tilde \chi_4}=2\psi(q).
$$
For the case of $c_{\mathbb L}$,
\begin{align}
\theta_{\tilde \chi_4}+\theta_{\chi_{12}} &= \sum_{n \in \Z } \left( \tilde \chi_4(n) + \chi_{12}(n) \right) q^{n^2}
\\
&=2 \times \sum_{k \in \Z }(-1)^kq^{(6k+1)^2}
\\
&=2c_{\mathbb L}(q).
\end{align}
Now $\chi_4$ is odd so $\theta_{\tilde \chi_4}=0$. In summary, we have
$$
\theta_{\chi_{12}}=2c_{\mathbb L}(q).
$$
\end{proof}
\subsection{Proof of the mirror identity}
Let $M_k(\Gamma)$ be the space of modular forms with level $\Gamma$ and weight $k$, with respect to some multiplier system. 
The desired identity \eqref{mirrorid}  becomes the equivalence of the following two series
\begin{align}\label{thetaid}
A(q) &= \theta_{\chi_{12}}(q)^2 [\Theta_{\chi_{4}}(q^9),\Theta_{\tilde \chi_4}(q)]_1 
\\
B(q) &=  12 \Theta_{\chi_{4}}(q^9) \left( 27 \Theta_{\chi_{4}}(q^9)^3 - \Theta_{\tilde \chi_4}(q)^3 \right).
\end{align}
\begin{remark}
Here we use the notation of Rankin-Cohen brackets \cite{R57} which are series of pairings
$$
[-,-]_n \colon M_k \times M_\ell \to M_{k+\ell+2n}
$$
for $n=0,1, \cdots$. When $n=0$, $[f,g]_0=fg$. When $n=1$, it is given by 
$$
[f,g]_1 = k fq \partial_q g - \ell g q\partial _q f.
$$
\end{remark}
From \cite{LZ}, $\phi, \psi$ has weight $\frac{3}{2}$ and one can check that $c_{\mathbb{L}}$ has weight $\frac{1}{2}$. Hence both side of \eqref{mirrorid} has weight $6$.

\begin{prop}
Each of $A(q)$ and $B(q)$ belongs to $M_6(\Gamma_0(576))$. 
\end{prop}
\begin{proof}
This is a direct application of Shimura's formula, Proposition~\ref{shimura}. It is clear from the outset that Shimura's formula would yield a transformation formula for $\Gamma_0(N)$ for some large $N$. Indeed, $\Theta_{\chi_4}(q)$ has a transformation formula for $\Gamma_0(4\cdot 4^2)= \Gamma_0(64)$ whence $\Theta_{\chi_4}(q^9)$ for $\Gamma_0(64\cdot9)=\Gamma_0(576)$. It follows, using \eqref{eq-old}, that $\Theta_{\tilde \chi_4}(q)$ satisfies a transformation formula for $\Gamma_0(576)$. Also, $\theta_{\chi_{12}}(q)$ satisfies a transformation formula for $\Gamma_0(4\cdot12^2)=\Gamma_0(576)$. This shows that one can take $N=576$.
\par
Now it remains to show that the multiplier system and the character are both trivial. It follows from \eqref{j}: $j(\gamma,\tau)^4 = (c\tau+d)^2$, and the fact that the characters $\chi_4$ and $\chi_{12}$ are quadratic. 
\end{proof}
\begin{prop}
If $A(q) = B(q)$ modulo $q^{577}$, then $A(q)=B(q)$. 
\end{prop}
\begin{proof}
It is a direct consequence of Sturm's bound \cite{Sturm}.
\end{proof}
Now it remains to check $A(q)=B(q)$ up to $576$-th order. We have checked it by computer
program up to millionth order.
 \end{proof}
 
\bibliographystyle{amsalpha}

\begin{thebibliography}{}
\bibitem{Aus} M. Auslander, {\em Functors and morphisms determined by objects}, Representation theory of algebras(Prof. Conf., Temple Univ., Philadelphia, Pa., 1976), Lecture Notes in Pure Appl. Math. Vol. 37, 1-244, 1978.


\bibitem{BC} P. Biran and O. Cornea, {\em Quantum structures for Lagrangian submanifolds}, preprint (2007), arXiv:0708.4221.
\bibitem{Cho} C.-H. Cho {\em Products of Floer cohomology of torus fibers in toric Fano manifolds,} Comm. Math. Phys. 260(2005), 613-640.
\bibitem{C} C.-H. Cho, {\em Strong homotopy inner products of an $\AI$-algebra,}  Int. Math. Res. Not., 2008 (2008), 35.
doi:10.1093/imrn/rnn041.
\bibitem{CO} C.-H. Cho and Y.-G. Oh
{\em Floer cohomology and disc instantons of Lagrangian torus fibers
in toric Fano manifolds}, Asian Journ. Math. 10 (2006), 773-814.


\bibitem{CHL} C.-H. Cho, H. Hong and S.-C. Lau, {\em Localized mirror functor for Lagrangian immersions, and homological mirror symmetry for $\mathbb{P}^1_{a,b,c}$}, J. Differential Geom. 106 (2017), no. 1, 45-126.

\bibitem{CHL2} C.-H. Cho, H. Hong and S.-C. Lau, {\em Localized mirror functor constructed from a Lagrangian torus}, preprint (2014), arXiv:1406.4597.

\bibitem{CHL3} C.-H. Cho, H. Hong and S.-C. Lau, {\em Noncommutative homological mirror functor}, preprint (2015), arXiv:1512.07128.
\bibitem{Cos} K. Costello, {\em Topological conformal field theories and Calabi-Yau categories}, Adv. Math. 210 (2007), no. 1, 165-214.
\bibitem{Du} B. Dubrovin, {\em Geometry of 2D topological field theories}, Lecture note in Math. 1620(1996), 120-348.
\bibitem{Dyc} T. Dyckerhoff, {\em Compact generators in categories of matrix factorizations,} Duke Math. J. 159 (2011), no. 2, 223-274.

\bibitem{DM} T. Dyckerhoff and D. Murfet, {\em The Kapustin-Li formula revisited}, Adv. Math. 231 (2012), no. 3-4, 1858-1885.


\bibitem{Fu} K. Fukaya, {\em Cyclic symmetry and adic convergence in Lagrangian Floer theory},
Kyoto J. Math. 50 (2010), no. 3, 521-590.

\bibitem{FOOO} K. Fukaya, Y.-G. Oh, H. Ohta and K. Ono, {\em Lagrangian intersection Floer theory: anomaly and obstruction. Part I}, AMS/IP Studies in Advanced Mathematics, vol. 46, American Mathematical Society, Providence, RI. 2009.

\bibitem{FOOOtoric1} K. Fukaya, Y.-G. Oh, H. Ohta and K. Ono, {\em Lagrangian Floer theory for compact toric manifolds I}, Duke Math. J. 151 (2010), no. 1, 23-174.

\bibitem{FOOOtoric2} K. Fukaya, Y.-G. Oh, H. Ohta and K. Ono, {\em Lagrangian Floer theory for compact toric manifolds II: bulk deformations}, Selecta Math. (N.S.) 17 (2011), no. 3, 609-711.

\bibitem{FOOOtoric3} K. Fukaya, Y.-G. Oh, H. Ohta and K. Ono, {\em Lagrangian Floer theory and mirror symmetry for compact toric manifolds}, Asterisque, No. 376 (2016), vi+340 pp.

\bibitem{G} S. Ganatra, {\em Symplectic cohomology and duality for the wrapped Fukaya category}, Thesis (Ph.D.)-Massachusetts Institute of Technology (2012).

\bibitem{GPS} S. Ganatra, T. Perutz and N. Sheridan, {\em Mirror symmetry: from categories to curve counts}, preprint (2015), arXiv:1510.03839.


\bibitem{KL} A. Kapustin and Y. Li {\em Topological correlators in Landau-Ginzburg models with boundaries}, Adv. Theor. Math. Phys. 7 (2003), no. 4, 727-749.

\bibitem{KS}  
M. Kontsevich and Y. Soibelman, {\em Notes on $\AI$-algebras, $\AI$-categories and non-commutative geometry}, Homological Mirror Symmetry. Lecture Notes in Physics, vol 757. Springer, Berlin, Heidelberg.


\bibitem{LZ} S.C. Lau, J. Zhou, {\em Modularity of open Gromov-Witten potentials of elliptic orbifolds},
Commun.Num.Theor.Phys. 09 (2015) 345-385

\bibitem{La} C. I. Lazariou, {\em On the structure of open-closed topological field theory in two dimensions}, Nuclear Phys. B 603 (2001), no. 3, 497-530.


\bibitem{Mur} D. Murfet, {\em Residues and duality for singularity categories of isolated Gorenstein singularities}, Compos. Math. 149 (2013), no. 12, 2071-2100.

\bibitem{Or} D. Orlov, {\em Triangulated categories of singularities and D-branes in Landau-Ginzburg models}, Proc. Steklov Inst. Math. 2004, no. 3 (246), 227-248.

\bibitem{PV10} A. Polishchuk and A. Vaintrob, {\em Chern characters and Hirzebruch-Riemann-Roch formula for matrix factorizations}, Duke Math. J. 161 (2012), no. 10, 1863-1926.

\bibitem{R57}  R.A. Rankin, {\em The construction of automorphic forms from the derivatives of given forms}, Michigan Math. J.  4 (1957), no. 2, 181-186

\bibitem{Sa} K. Saito, {\em Period mapping associated to a primitive form}, Publ. R.I.M.S. 19 (1983), 1231-1261.

\bibitem{ST} I. Satake and A. Takahashi, {\em Gromov-Witten invariants for mirror orbifolds of simple elliptic singularities}, Annales de l'institut Fourier (2011), vol 61, Issue: 7, 2885-2907.


\bibitem{Seg} G. B. Segal, {\em The definition of conformal field theory}, In: Differential geometrical methods in theoretical physics (Como, 1987), NATO Adv. Sci. Inst. Ser. C Math. Phys. Sci., 250, 
Kluwer Acad. Publ., Dordrecht, 1988, pp. 165-171.

\bibitem{Se1} P. Seidel, {\em Fukaya categories and deformations},  Proc. International Congress of Mathematicians,
vol. II (Beijing, 2002), Higher Ed. Press, Beijing, 351-360.

\bibitem{Seidelbook} P. Seidel, {\em Fukaya categories and Picard-Lefschetz theory}, Zurich Lectures in Advanced Mathematics. European Mathematical Society (EMS), Z\"urich, 2008. viii+326 pp.
 

\bibitem{Se2} P. Seidel, {\em Homological mirror symmetry for genus two curve}, J. Algebraic Geom. 20 (2011), 727-769.


\bibitem{Sh} N. Sheridan, {\em Formulae in noncommutative Hodge theory}, preprint (2015), arXiv:1510.03795.

\bibitem{Shim} G. Shimura, {\em On modular forms of half-integral weight},
Annals of Math. Second Series, Vol. 97, No. 3 (1973),  440-481


\bibitem{Shk} D. Shklyarov, {\em Calabi-Yau structures on categories of matrix factorizations}, Journal of Geometry and Physics 119 (2017), 193-207.

\bibitem{T} T. Tradler, {\em Infinity-inner-products on A-infinity-algebras}, J. Homotopy Relat. Struct., 3 (2008), 245-271.

\bibitem{Sturm} J. Sturm, {\em On the congruence of modular forms}, Lecture Notes in Mathematics 1240 (1987), 275-280.
 

\end{thebibliography}

\end{document}